\documentclass[11pt, oneside]{article}   	
\usepackage{geometry}                		
\usepackage{graphicx}				
\usepackage{amssymb}
\usepackage{hyperref}

\usepackage{amsmath}
\usepackage{amssymb}
\usepackage{latexsym}
\usepackage{amsthm}

\newcommand{\dqbin}[3]{\displaystyle\genfrac{[}{]}{0pt}{}{#1}{#2}_{#3}}

\newcommand{\C}{\mathbf{C}}
\newcommand{\PP}{\mathbf{P}}
\newcommand{\R}{\mathbf{R}}
\newcommand{\I}{\mathbf{I}}

\newcommand{\e}{\mathbf{e}}
\newcommand{\vv}{\mathbf{v}}

\newcommand{\LL}{\mathbf{L}}
\newcommand{\GG}{\mathbf{G}}
\newcommand{\anti}{\mathbf{A}}

\newcommand{\integers}{\mathbb{Z}}
\newcommand{\reals}{\mathbb{R}}
\newcommand{\Sn}{{\mathbf S}_n}
\newcommand{\des}{{\rm des}\,}
\newcommand{\cyc}{{\rm cyc}\,}
\newcommand{\exc}{{\rm exc}\,}

\newcommand{\asc}{{\rm asc}\,}
\newcommand{\bin}{{\rm bin}\,}
\newcommand{\binv}{{\rm binv}\,}
\newcommand{\sq}{{\rm sq}\,}
\newcommand{\sqin}{{\rm sqin}\,}
\newcommand{\siz}{{\rm siz}\,}
\newcommand{\lhp}{{\rm lhp}\,}
\newcommand{\s}{\mathbf{s}}

\newcommand{\amaj}{{\rm amaj}\,}

\newcommand{\D}{{\rm Des}\,}
\newcommand{\A}{{\rm Asc}\,}
\newcommand{\inv}{{\rm inv}\,}
\newcommand{\maj}{{\rm maj}\,}
\newcommand{\comaj}{{\rm comaj}\,}
\newcommand{\la}{\lambda}
\newcommand{\card}[1]{\ensuremath{\left|#1\right|}}
\newcommand{\ceil}[1]{\ensuremath{\left\lceil #1 \right\rceil}}
\newcommand{\floor}[1]{\ensuremath{\left\lfloor #1 \right\rfloor}}
\newcommand{\qint}[1]{\ensuremath{\left[\hspace{.02in} #1 \hspace{.02in}\right]}}

\newtheorem{theorem}{Theorem}[section]
\newtheorem{corollary}[theorem]{Corollary}

\newtheorem{observation}[theorem]{Observation}

\title{The Mathematics of Lecture Hall Partitions}
\author{
Carla D. Savage\\
 Department of Computer Science\\ North Carolina State University\\
Raleigh, NC 27695-8206, USA\\
 \texttt{savage@ncsu.edu}\\
 }
\date{May 1, 2016}							

\begin{document}
\maketitle
\begin{abstract}
Over the past twenty years, lecture hall partitions have emerged as fundamental  combinatorial structures,
leading to new generalizations and interpretations of classical theorems and new results.  In recent years, geometric approaches to lecture hall partitions have used polyhedral geometry to discover further properties of these rich combinatorial objects.

In this paper we give an overview of some of the surprising connections that have surfaced in the process of trying to understand the lecture hall partitions.

\end{abstract}


	\tableofcontents

\pagebreak
\section{The Lecture Hall Theorem}
\label{section:LHT}

A {\em partition} $p=(p_1, p_2, \ldots, p_k)$ of an integer $N$ is an unordered collection of positive integers  $p_i$, called {\em parts}, such that $|p|=p_1 + p_2 + \ldots +p_k= N$. Since the order of the parts does not matter, it is convenient  to list the parts in
nonincreasing order $p_1 \geq p_2 \geq \ldots \geq p_k$.  The number of partitions of $N$ is the coefficient of $q^{N}$ in 
$1/(q;q)_{\infty}$, where we use the notation
$(a;q)_{n} = \prod_{i=0}^{n-1} (1-aq^i)$,
$(a;q)_{\infty} = \prod_{i=0}^{\infty} (1-aq^i)$, and $(a_1,\ldots,a_k;q)_{\infty} = \prod_{j=i}^{k} (a_j;q)_{\infty} $.

Perhaps the best known theorem involving partitions is the following.

\begin{theorem}[{\bf Euler's partition theorem}]
The number of  partitions of $N$  in which all parts are distinct 
is equal to the number of partitions of $N$ in which all parts are odd.
\label{thm:epthm}
\end{theorem}
The analytic form of Theorem \ref{thm:epthm} is 
\begin{eqnarray}
\prod_{i \geq 1}(1+q^i) & = & \prod_{i \geq 1}\frac{1}{1-q^{2i-1}},
\label{eptgf}
\end{eqnarray}
where the left-hand side is the sum of $q^{|p|}$ over partitions $p$ into distinct parts and the right-hand side is 
the sum of $q^{|p|}$ over partitions into odd parts.

The {\em lecture hall partitions} were introduced by Bousquet-M\'elou and Eriksson  in \cite{BME1} and defined by:
\begin{align}
\LL_n =  \left \{\la \in \integers^n \ \biggm\vert  \  0 \leq \frac{\la_{1}}{{1}}
\leq \frac{\la_{2}}{{2}} \leq \cdots
\leq \frac{\la_{n}}{{n}}  \right \}.
\end{align}
The conditions on $\la_i$ can be viewed as height constraints on the $i$th row of a lecture hall in order that each student can see over the head of the students in previous rows  to view the shoes of her teacher.

In 1997, Bousquet-M{\'e}lou and Eriksson proved the following.
\begin{theorem}[{\bf The lecture hall theorem} \cite{BME1}]
The number of lecture hall partitions of $N$  in $\LL_{n}$
is equal to the number of partitions of $N$ into odd parts less than $2n$.
\label{thm:lhthm}
\end{theorem}The generating function version of Theorem \ref{thm:lhthm} is:
\begin{eqnarray}
\sum_{\la \in \LL_n} q^{|\la|} \  &  = & 
\prod_{i=1}^{n} \frac{1}{1-q^{2i-1}}.
\label{lhpgf}
\end{eqnarray}

If we reverse a lecture hall partition and ignore the parts of size 0, we do actually get a {\em partition} into distinct parts.
In this sense, as $n \rightarrow \infty$, $\LL_n$ becomes the set of partitions into distinct parts. Theorem \ref{thm:lhthm}, therefore, is a finite version of Theorem \ref{thm:epthm}, and its discovery came as an unexpected surprise when Bousquet-M\`elou and Eriksson introduced it in 1997.

Whereas Theorem \ref{thm:epthm} is well-understood analytically, combinatorially, and algebraically, 
Theorem \ref{thm:lhthm} is hardly understood at all.  This is in spite of the fact that by now there are many proofs, including those of Bousquet-M\`elou and Eriksson \cite{BME1,BME2,BME3}, Andrews \cite{Andrews}, Yee
\cite{Yee1,Yee2},  Andrews, Paule, Riese, and Strehl \cite{PA5}, Eriksen \cite{eriksen2002simple}, and Bradford et al \cite{abacus}.
We have also contributed  to the collection of proofs with co-authors 
Corteel \cite{truncated},
Corteel and Lee \cite{CLS},
Andrews and Corteel \cite{ACS},
Bright \cite{BS},
and, most recently,
Corteel and Lovejoy \cite{CLovejoyS}.

In the search for a transparent proof for the lecture hall theorem, many connections have been discovered.  Over the past twenty years, lecture hall partitions have emerged as fundamental structures in combinatorics, number theory, algebra, and geometry,  leading to new generalizations and interpretations of classical theorems and new results.  

In this paper we give an overview of some of the surprising connections that have surfaced and results that have been discovered in the process of trying to understand lecture hall partitions.

\section{$\s$-lecture hall partitions}
\label{slhp}
For any sequence $\s$ of positive integers, define the {\em $s$-lecture hall partitions} by
\begin{align}
\LL_n^{(\s)} =  \left \{\la \in \integers^n \ \biggm\vert  \  0 \leq \frac{\la_{1}}{{s_1}}
\leq \frac{\la_{2}}{{s_2}} \leq \cdots
\leq \frac{\la_{n}}{{s_n}}  \right \}.
\end{align}
The name is a slight abuse of notation, since unless $\s$ is nondecreasing, $\la \in \LL_n^{(\s)}$ may not be monotone and therefore  may be a composition  of $|\la|$ rather than a partition.
 For example, $(2,3)$ and $(3,2)$ are distinct ``$(5,3)$-lecture hall partitions'' of 5 since
 $0 \leq 2/5 \leq 3/3$ and $0 \leq 3/5 \leq 2/3$.

The $\s$-lecture hall partitions were first considered by Bousquet-M\`elou and Eriksson in \cite{BME2}, where $\s$ was required to be nondecreasing.  They called a sequence $\s$ {\em polynomic} if 
\begin{eqnarray}
\sum_{\la \in \LL_n^{(\s)}} q^{|\la|} & = &  \prod_{i=1}^{n} \frac{1}{1-q^{d_i}}
\end{eqnarray}
for some positive integers $d_1, \ldots, d_n$.

In \cite{BME2} Bousquet-M\`elou and Eriksson discovered an infinite family of polynomic sequences - the so-called $(k,\ell)$-sequences - which are discussed in Section \ref{section:k_ell}.

In Sections \ref{section:anti} and \ref{section:truncated} we show that there are some sequences $\s$ which are not polynomic, but nevertheless give rise to interesting generating functions.

\section{Anti-lecture hall compositions}
\label{section:anti}

In the search for  sequences $\s$ for which the $\s$-lecture hall partitions have an interesting generating function, it is natural to consider $\s=(n,n-1, \ldots,1)$.

Define the {\em anti-lecture hall compositions} $\anti_n$ by
\begin{eqnarray}
\anti_n &= & \left \{\la \in \integers^n \ \biggm\vert  \  0 \leq \frac{\la_{1}}{{n}}
\leq \frac{\la_{2}}{{n-1}} \leq \cdots
\leq \frac{\la_{n}}{{1}}  \right \}.
\end{eqnarray}
So $\anti_n = \LL_n^{(n,n-1, \ldots, 1)}$. For example $(1,2,3,4)$, $(4,3,2,1)$, and $(1,4,3,2)$ are all  in $\anti_4$, but $(5,3,2,1)$ is not. The generating function for $\anti_n$  has a simple product form.
\begin{theorem}[{\bf The  anti-lecture hall theorem} (Corteel, S \cite{anti})]
\begin{eqnarray}
A_n(q) := \sum_{\la \in \anti_n} q^{|\la|} &  = &
\prod_{i=1}^{n} \frac{1+q^i}{1-q^{i+1}} \ = \ \frac{(-q;q)_n}{(q^2;q)_n}.
\label{eq:antigf}
\end{eqnarray}
\label{thm:antigf}
\end{theorem}
The  product in \eqref{eq:antigf}  suggests a connection with the  {overpartitions} \cite{CL}  of Corteel and Lovejoy.
An {\em overpartition} of $N$  is a nonincreasing sequence of positive numbers whose sum is $N$   in which the first occurrence of a number may be overlined.  Clearly, summing over all overpartitions $\la$,
 \begin{eqnarray}
 \sum_{\la}q^{|p|} &= &\prod_{i=1}^{\infty}\frac{1+q^i}{1-q^i} = \frac{(-q;q)_{\infty}}{(q;q)_{\infty}}.
 \label{overpartition_gf}
 \end{eqnarray}

Comparing \eqref{overpartition_gf} with  \eqref{eq:antigf}, letting $n \rightarrow \infty$ in \eqref{eq:antigf},  and ignoring the `0' parts, we have the following.
\begin{observation}
The number of anti-lecture hall compositions of $N$ (of any length) is equal to the number of overpartitions of $N$ with no un-overlined 1.  
\end{observation}

The first hint that there might be a deeper connection between anti-lecture hall compositions and overpartitions is this surprising result of Chen, Sang and Shi \cite{chen}.
Let $\mathcal{A}_t$ be the set of anti-lecture hall compositions (of any length $k$)  into positive parts 
with last part at most $t$:
\[
\mathcal{A}_t: \ \ 0 \ < \ \frac{\la_1}{k} \leq \frac{\la_2}{k-1} \leq \ldots \leq \frac{\la_k}{1} \leq t.
\]
\begin{theorem}[Chen, Sang, Shi \cite{chen}]
The number of anti-lecture hall compositions  of $N$ in $\mathcal{A}_t$ 
is the same as the number of overpartitions of $N$ with no un-overlined parts congruent to $0, \pm 1$ modulo $t+2$.  That is,
\begin{eqnarray}
\sum_{\la \in \mathcal{A}_t} q^{|\la|} & = &
\frac{(-q;q)_{\infty}}{(q;q)_{\infty}} 
(q,q^{t+1},q^{t+2};q^{t+2})_\infty.
\end{eqnarray}
\label{thm:chen}
\end{theorem}

An alternative approach and a  generalization of Theorem \ref{thm:chen}, is given with Corteel and Lovejoy in \cite{CLovejoyS}.  The method of  \cite{CLovejoyS} derives a combinatorial recurrence for a (finite) version of a 3-parameter refinement and solves it using Andrews' generalization of the Watson-Whipple transformation.

In  \cite{stamps}, Engstr{\"o}m and Stamps
make use of anti-lecture hall compositions to give a geometric picture of Betti diagrams of ideals with 2-linear resolutions and thereby  to prove that any Betti diagram of a module with a 2-linear resolution arises from a direct sum of Stanley-Reisner rings constructed from threshold graphs.

\section{Truncated lecture hall partitions}
\label{section:truncated}

Computations suggested that limiting the number of positive parts of $\la \in \LL_n$  would give rise to an interesting refinement of Theorem \ref{thm:lhthm}.  Let $\LL_{n,k}$ be the set of $\la \in \LL_n$ with at most $k$ positive parts.  The elements of $\LL_{n,k}$ are {\em truncated lecture hall partitions}.
Note that $\LL_{n,k}=\LL_k^{(n-k+1, n-k+2, \ldots, n)}$.

\begin{theorem}[{\bf The truncated lecture hall theorem} (Corteel, S \cite{truncated})]
The number of lecture hall partitions of $N$  in $\LL_{n,k}$
is equal to the number of partitions of $N$ into odd parts less than $2n$, with the following constraint on the parts:  at most $\floor{k/2}$ of the odd parts can be chosen from the interval $[2\ceil{k/2}+1, \ 2(n-\ceil{k/2})-1]$.
\label{thm:truncated}
\end{theorem}
The proof in \cite{truncated} uses $q$-calculus.  A combinatorial proof  of Theorem \ref{thm:truncated} appears in Appendix 1 of \cite{SY}.

For a nonnegative integer $n$, let $\qint{n}_q= \frac{1-q^n}{1-q}$ and
 $\qint{n}_q! = \prod_{i=1}^n \qint{i}_q.$  The {\em $q$-binomial coefficient }is defined by
\begin{eqnarray*}
\dqbin{n}{k}{q} & = &  \frac{\qint{n}_q! }{\qint{k}_q! \qint{n-k}_q!}.
\end{eqnarray*}
Let $\anti_{n,k}$ be the set of {\em truncated anti-lecture hall compositions} defined by
$\anti_{n,k} = \LL_k^{(n,n-1, \ldots, n-k+1)}$. 
Let ${\overline{\LL}}_{n,k}$ be the set of $\la \in \LL_{n,k}$ with {\em exactly} $k$ positive parts.
The generating functions for  ${\overline{\LL}}_{n,k}$ and $\anti_{n,k}$
have the following product forms.
\begin{theorem}[Corteel, S \cite{truncated}]
\begin{eqnarray}
\sum_{\la \in {\overline{\LL}}_{n,k}} q^{|\la|}  & = &q^{{k+1 \choose 2}} \dqbin{n}{k}{q}
\frac{(-q^{n-k+1};q)_k}{(q^{2n-k+1};q)_k}; \label{Ltrunc}\\
\sum_{\la \in \anti_{n,k}} q^{|\la|} & = & \dqbin{n}{k}{q}\frac{(-q^{n-k+1};q)_k}{q^{2(n-k+1)};q)_k}.
\label{Atrunc}
\end{eqnarray}
\label{thm:q_truncated}
\end{theorem}
In many cases, enumerating lecture hall partitions can be accomplished using combinatorial reasoning to devise a recurrence for a multi-parameter refinement, followed by application of the right $q$-series identity to solve the recurrence. Andrews' generalization of the Watson-Whipple transformation was mentioned in this context in Section \ref{section:anti}.
This was  also the case for the identities  in Theorem \ref{thm:q_truncated} which were proven in \cite{truncated} with the help of the $q$-Chu Vandermonde identities:
\[
\frac{a^n(c/a;q)_n}{(c;q)_n} \ = \ \sum_{m=0}^n \frac{(a;q)_m(q^{-n};q)_m}{(c;q)_m(q;q)_m}q^m
\]
\[
\frac{(c/a;q)_n}{(c;q)_n} \ = \ \sum_{m=0}^n  \dqbin{n}{m}{q} \frac{(a;q)_m}{(c;q)_m}(-c/a)^mq^{{m \choose 2}}.
\]
A proof of Theorem \ref{thm:q_truncated} appearing in \cite{CLS} derives different recurrences, but solves them with the help of the same $q$-Chu Vandermonde identities.
Yet another pair of recurrences were derived in \cite{FiveG} and solved in \cite{ACS} using two new finite corollaries of a $q$-analog of Gauss's second theorem.

\section{The $(k,\ell)$-lecture hall theorem}
\label{section:k_ell}

Define
the {\em $(k,\ell)$ sequence } $a^{(k,\ell)}$ by 
\begin{eqnarray}
a^{(k,\ell)}_{2i}  &= &\ell a^{(k,\ell)}_{2i-1}-a^{(k,\ell)}_{2i-2} \label{even} \\
a^{(k,\ell)}_{2i+1}  &=& k a^{(k,\ell)}_{2i}-a^{(k,\ell)}_{2i-1},
\label{def:k_ell}
\end{eqnarray}
with initial conditions $a^{(k,\ell)}_1=1$, $a^{(k,\ell)}_2=\ell$.  When $k=\ell=2$, 
$a^{(k,\ell)}$ is the sequence of positive integers.

In \cite{BME2} Bousquet-M\`elou and Eriksson show  that the generating function for the $a^{(k,\ell)}$-lecture hall partitions is a natural generalization of \eqref{lhpgf}.  We'll describe the result, sketch the proof, and show how it gives rise to a recursive bijection.  In the special case $k=\ell=2$, this is a proof (our favorite) of the lecture hall theorem.

For this section it will be convenient to reverse the labeling of the parts of an $a^{(k,\ell)}$-lecture hall partition.  
To avoid confusion, define
\[
\GG_n^{(k,\ell)}= \left \{ \la \in \integers^n \ \biggm\vert  \  \frac{\la_1}{a_n^{(k,\ell)}} \geq \frac{\la_2}{a_{n-1}^{(k,\ell)}} \geq \ldots \geq
\frac{\la_n}{a_{1}^{(k,\ell)}} \geq 0 \right \}
\]
For $\la \in \GG_n^{(k,\ell)}$define $|\la|_o= \la_1 + \la_3+ \cdots$ and
$|\la|_e= \la_2+\la_4 + \cdots$. Define
\[
G_n^{(k,\ell)}(x,y) =\sum_{\la \in \GG_n^{(k,\ell)}} x^{|\la|_o}y^{|\la|_e}.
\]
Bousquet-M\`elou and Eriksson
proved the following theorem for $k \geq 2$ and $\ell \geq 2$.
It is shown in \cite{CortSavSills} that the result holds as well for $(k,\ell) \in \{(1,4),\ (4,1) \}$.

\begin{theorem}[{\bf The $(k,\ell)$-lecture hall theorem} (Bousquet-M\`elou, Eriksson \cite{BME2})] For positive  integers $k, \ell$ with $k\ell \geq 4$,
\[
G_{2n}^{(k,\ell)}(x,y) = 
\prod_{i=1}^{2n} \frac{1}
{1-x^{a^{(k,\ell)}_i }y^{a_{i-1}^{(\ell,k)}} } \\
\]
\[
G_{2n-1}^{(k,\ell)}(x,y) = \prod_{i=1}^{2n-1} \frac{1}
{1-x^{a^{(\ell,k)}_i }y^{a_{i-1}^{(k,\ell)}} }. 
\]
\label{thm:k_ell}
\end{theorem}
\begin{proof}
We outline the clever approach of Bousquet-M\`elou and Eriksson. 
The description below is adapted from 
\cite{CortSavSills}, Appendix 1.

The strategy is to show that the following recurrence holds:
\begin{eqnarray}
G_n^{(k,\ell)}(x,y) & = & G_{n-1}^{(k,\ell)}(x^{\ell}y,x^{-1})/(1-x) \ \ \ \ {\rm if} \ n \ {\rm is \ even} \label{evenrec}\\
G_n^{(k,\ell)}(x,y) & = & G_{n-1}^{(k,\ell)}(x^{k}y,x^{-1})/(1-x) \ \ \ \ {\rm if} \ n \ {\rm is \ odd},
\label{oddrec}
\end{eqnarray}
with initial condition $G_0^{(k,\ell)}(x,y)=1$.  It can then be shown by induction, using the properties of $a^{(k,\ell)}$,
that the theorem follows.

Recurrences \eqref{evenrec} and \eqref{oddrec} are  derived as follows. To simplify notation, let $a_n=a_n^{(k,\ell)}$.  Define
\begin{eqnarray}
\Gamma_n:\ \ \GG_{n-1}^{(k,\ell)} \times \mathbb{N}& \rightarrow& \GG_{n}^{(k,\ell)}.
\label{gamma}
\end{eqnarray}
  Given $(\la,s) \in \GG_{n-1}^{(k,\ell)} \times \mathbb{N}$,  let
 $\Gamma_n(\la, s) = \mu$, where
\begin{eqnarray*}
\mu_1 &  = & \ceil{\frac{a_n\la_1}{a_{n-1}}}+s\\
\mu_{2t} & = & \la_{2t-1}, \ \ 1 \leq t \leq n/2;\\
\mu_{2t+1} & = & \left\{ 
\begin{array}{ll}
\ceil{\frac{a_{n-2t}\la_{2t+1}}{a_{n-2t-1}}} + \floor{\frac{a_{n-2t}\la_{2t-1}}{a_{n-2t+1}}} - \la_{2t}, \ & 1 \leq t < (n-1)/2\\
\\
\floor{\frac{a_{n-2t}\la_{2t-1}}{a_{n-2t+1}}} - \la_{2t}& t=(n-1)/2.
\end{array}
\right .
\end{eqnarray*}
One then uses the properties of  $a^{(k,\ell)}$  
to prove that $\mu \in \GG_{n}^{(k,\ell)}$, that $\Gamma_n$ is a bijection, and that 
\begin{eqnarray*}
|\mu|_e & = & |\la|_o;\\
|\mu|_o & = & 
\left \{
\begin{array}{ll}
\ell|\la|_o-|\la|_e+s & {\rm if} \ n  \ {\rm is  \ even}\\
k|\la|_o-|\la|_e+s & {\rm if} \ n  \ {\rm is  \ odd}.
\end{array}
\right.
\end{eqnarray*}
Then, when $n$ is even,
\begin{eqnarray*}
G_n^{(k,\ell)}(x,y)  = \sum_{\mu \in \GG_n^{(k,\ell)}} x^{|\mu|_o}y^{|\mu|_e}
& = &\sum_{\la \in \GG_{n-1}^{(k,\ell)}} \ \sum_{s=0}^{\infty}x^{\ell|\la|_o-|\la|_e +s}y^{|\la|_o}\\
& = & \frac{1}{1-x}\sum_{\la \in \GG_{n-1}^{(k,\ell)}}(x^{\ell}y)^{|\la|_o}((1/x)^{|\la|_e}\\
& = & \frac{G_{n-1}^{(k,\ell)}(x^{\ell}y,x^{-1})}{1-x},
\end{eqnarray*}
giving \eqref{evenrec}.  When $n$ is odd, similar reasoning gives \eqref{oddrec}. 
\end{proof}

Fix $k,\ell$ and let
$\rho_i=a^{(k,\ell)}_i +a_{i-1}^{(\ell,k)}$ and $r_i=a^{(\ell,k)}_i +a_{i-1}^{(k,\ell)}$, with $\rho_1=r_1=1$.
As noted in \cite{BME1,BME2}  the proof of Theorem \ref{thm:k_ell} gives rise to a recursive weight-preserving bijection that we refer to  here as BME:
\begin{eqnarray}
{\rm BME}_n: \GG_{n}^{(k,\ell)}& \rightarrow &
\left \{
\begin{array}{ll}
{\mbox{partitions into parts from  $\{\rho_1, \ldots \rho_n\}$}} & {\mbox{if $n$ is even}} \\
{\mbox{partitions into parts from  $\{r_1, \ldots r_n\}$}} & {\mbox{if $n$ is odd}}.
\end{array}
\right .
\label{BMEn}
\end{eqnarray}
Represent a partition $\alpha$ into parts  from  $\{\rho_1, \ldots \rho_n\}$
as
 $\alpha = \rho_n^{m_n} \rho_{n-1}^{m_{n-1} }\cdots \rho_1^{m_1}$
 where $m_i$ is the number of copies of $\rho_i$ in $\alpha$.  Do likewise when the parts are chosen
 instead from $\{r_1, \ldots r_n\}$.

\noindent
For $\mu \in \GG_{n}^{(k,\ell)}$, define  BME$_n(\mu)$  as follows, using $\Gamma_n$ from \eqref{gamma}.

\noindent
If $n=0$, $\mu$ is the empty partition and so is BME$_n(\mu)$. Otherwise,

\noindent
if $\mu = \Gamma_n(\la, s)$ then

{\bf $n$ even:} if ${\rm BME}_{n-1}(\la) = r_{n-1}^{m_{n-1}} r_{n-2}^{m_{n-2} }\cdots r_1^{m_1}$
then ${\rm BME}_n(\mu)=\rho_{n}^{m_{n-1}} \rho_{n-1}^{m_{n-2} }\cdots \rho_2^{m_1} \rho_1^s$;

{\bf $n$ odd:} if ${\rm BME}_{n-1}(\la) = \rho_{n-1}^{m_{n-1}} \rho_{n-2}^{m_{n-2} }\cdots \rho_1^{m_1}$
then ${\rm BME}_n(\mu)=r_{n}^{m_{n-1}} r_{n-1}^{m_{n-2} }\cdots r_2^{m_1} r_1^s$.

\noindent
BME is simple to implement and useful for experiments. As an example:
\begin{eqnarray*}
\Gamma_2((0),4) & = & (4,0) \in \GG_{2}^{(1,4)}\\
\Gamma_3((4,0),1) & = & (4,4,1) \in \GG_{3}^{(1,4)}\\
\Gamma_4((4,4,1),1) & = & (12,4,5,1) \in \GG_{4}^{(1,4)}\\
\Gamma_5((12,4,5,1),1) & = & (9,12,4,5,0) \in \GG_{5}^{(1,4)}
\end{eqnarray*}
and so, since 
$(\rho_1, \ldots, \rho_5) = (1,5,4,11,7)$ and
$(r_1, \ldots, r_5) = (1,2,7,5,13)$,
\begin{eqnarray*}
{\rm BME}_1(0) & = & 1^0\\
{\rm BME}_2(4,0) & = & 5^0 1^4\\
{\rm BME}_3(4,4,1) & = & 7^0 2^4 1^1\\
{\rm BME}_4(12,4,5,1) & = & 11^0 4^4 5^1 1^1\\
{\rm BME}_5(9,12,4,5,0) & = & 13^0 5^4 7^1 2^1 1^1.\\
\end{eqnarray*}

\section{A generalization of Euler's partition theorem}
\label{ell_Euler}

Define the {\em $\ell$-sequence} ${ a}^{(\ell)}$ to be the $(k, \ell)$-sequence of \eqref{even}, \eqref{def:k_ell}
with $k=\ell$.  Then
\[
a_n^{(\ell)} =  \ell a_{n-1}^{(\ell)} -a_{n-2}^{(\ell)}, 
\]
with 
$a_1^{(\ell)}=1$, $a_2^{(\ell)} =\ell$.  In this case, Theorem \ref{thm:k_ell} takes the following form.
\begin{theorem}[{\bf The $\ell$-lecture hall theorem} (Bousquet-M{\'e}lou, Eriksson \cite{BME2})]
For $\ell \geq 2$ and $n \geq 0$,
\[
\sum_{\la} q^{\la_1+\la_2 + \ldots + \la_n} \ = \prod_{i=1}^n \frac{1}{ \ 1-q^{a^{(\ell)}_{i}+a^{(\ell)}_{i-1}}}
\]
where the sum is over all integer sequences $\la=(\la_1, \ldots, \la_n)$ satisfying
\[
\frac{\quad \la_1 \quad}{a_n^{(\ell)}} \ \geq \ \frac{\quad \la_2 \quad }{a_{n-1}^{(\ell)}} \ \geq \ \ldots\  \geq\ 
\frac{\quad \la_n \quad}{a_{1}^{(\ell)}} \geq 0.
\]
\label{ell_lh}
\end{theorem}
Let $c_{\ell}$ be the largest root of the characteristic polynomial $x^2-\ell x + 1$.
Since $\lim_{n \rightarrow \infty} a_n^{(\ell)}/ a_{n-1}^{(\ell)} = c_{\ell}$, 
 observe that
 as $n \rightarrow \infty$, Theorem \ref{ell_lh} becomes the following generalization of Theorem \ref{thm:epthm}.
\begin{theorem}[{\bf The $\ell$-Euler theorem} (Bousquet-M{\'e}lou, Eriksson \cite{BME2})]
The number of partitions of an integer $N$ into parts from the set
\[
\{a_i^{(\ell)}+ a_{i-1}^{(\ell)} \ \bigm\vert  \ i \geq 1 \}
\] 
is the same as the number of partitions $\la = (\la_1, \la_2, \ldots)$ of $N$ in which $\la_i/\la_{i+1} > c_{\ell}$ for  all consecutive positive parts $\la_i, \la_{i+1}$.'
\label{thm:ell_euler}
\end{theorem}
\noindent 
When $\ell=2$, $a^{(\ell)}$ is the sequence of positive integers,
$a_i^{(\ell)}+ a_{i-1}^{(\ell)}=2i-1$ and $c_{\ell}=1$.  So  Theorem \ref{ell_lh} becomes
the lecture hall theorem
and Theorem \ref{thm:ell_euler}  becomes Euler's partition theorem.

In \cite{SY}, with Yee, we provide a simple bijection $\Theta^{(\ell)}$ to prove Theorem \ref{thm:ell_euler} .  The bijection $\Theta^{(\ell)}$ has several nice properties:
\begin{itemize}
\item $\Theta^{(2)}$ is Sylvester's bijection for Euler's partition theorem (Theorem \ref{thm:epthm}).
\item A finite version $\Theta_n^{(\ell)}$ of $\Theta^{(\ell)}$  proves the $\ell$-lecture hall theorem (Theorem \ref{ell_lh}).
\item $\Theta_n^{(2)}$ is  the same bijection for  Theorem  \ref{thm:lhthm} as the one due to Yee  in  \cite{Yee1}.
\item $\Theta_n^{(2)}$ allows for a combinatorial proof of  Theorem \ref{thm:truncated},
the truncated lecture hall theorem.
\item  It appears that $\Theta_n^{(\ell)}={\rm BME}_n$ \eqref{BMEn} for $k=\ell$, but this has not been proven.
\end{itemize}

We describe $\Theta^{(\ell)}$, to show how simple it is.
 Fix $\ell \geq 2$ and let  
$a_i = a_i^{(\ell)}$.  For positive integer $i$, let $p_i = a_i + a_{i-1}$.
 Let $O^{(\ell)} $ be the set of partitions (of any integer) into parts from the set 
 $\{p_1,  p_2, \ldots\}$. Represent a partition $\mu \in O^{(\ell)}$ as
 $\mu = p_r^{m_r} p_{r-1}^{m_{r-1} }\cdots p_1^{m_1}$
 where $m_i$ is the number of copies of $p_i$ in $\mu$.

Let $D^{(\ell)} $ be the set of partitions (of any integer) into a decreasing sequence of positive parts with the property the ratio of consecutive positive parts is at least $c_{\ell}$.

\noindent
{\bf Bijection for the $\ell$-Euler theorem} (S, Yee, \cite{SY})
$\ \ \ 
\Theta^{(\ell)}: \ O^{(\ell)} \ \rightarrow \ D^{(\ell)}
$

\noindent
For $\mu = p_r^{m_r} p_{r-1}^{m_{r-1} }\cdots p_1^{m_1}\in O^{(\ell)}$, define $\Theta^{(\ell)}(\mu)$ to be the sequence $\la=(\la_1, \la_2, \ldots)$ obtained from the empty partition $(0,0, \ldots)$ by inserting the parts of $\mu$ in nonincreasing order, one at a time, largest to smallest,
according to the following $\Theta^{(\ell)}$ insertion procedure.

\noindent
To insert $p_k$ into $(\la_1, \la_2, \ldots)$:

\vspace{-.1in}
If $k=1$ then $\la_1 \leftarrow \la_1+a_1$;

\vspace{-.1in}
otherwise, if $\la_1+a_k-a_{k-1} \  >  \ c_{\ell}(\la_2+a_{k-1}-a_{k-2})$ then

\vspace{-.1in}
\hspace{.5in} 
$\la_1 \leftarrow \la_1+a_k-a_{k-1}$;

\vspace{-.1in}
\hspace{.5in} 
$\la_2 \leftarrow \la_2+a_{k-1}-a_{k-2}$;

\vspace{-.1in}
\hspace{.5in} 
recursively insert $p_{k-1}$ into $(\la_3, \la_4, \ldots)$ via $\Theta^{(\ell)}$ insertion;

\vspace{-.1in}
otherwise,  $\la_1 \leftarrow \la_1+a_k$;   $\la_2 \leftarrow \la_2+a_{k-1}$;

\noindent
The most mysterious thing involved in proving  that $\Theta^{(\ell)}$ is a bijection is this:  If we follow this insertion procedure how do we know that $\la_2 > c_{\ell} \la_3$?  It  is {\em true}, but surprisingly difficult to prove. 
The proof method involves encoding integers $i$  as words  $w(i)$ with the property that 
$i > c_{\ell}j$ if and only if  $w(i) > w(j)$ in lexicographic order.

\noindent
{\bf Question:}  Is there an easier proof that $\Theta^{(\ell)}$ is a bijection?

\noindent
The bijection 
$\Theta_n^{(\ell)}: \  O_n^{(\ell)} \ \rightarrow \ D_n^{(\ell)}
$
for the $\ell$-lecture hall theorem, Theorem \ref{ell_lh}, involves a slight modification of $\Theta^{(\ell)}$, and can be found in \cite{SY}.

\section{$(1,4)$- and $(4,1)$- sequences and the little G{\"o}llnitz identities}
\label{Gollnitz}

Consider the following $(k,\ell)$-sequences defined in Section \ref{section:k_ell}:
\begin{eqnarray}
a^{(1,4)}  & = &  1,\ 4,\ 3,\ 8,\ 5,\ 12, \ 7, \ 16, \ 9, \ldots;\\
a^{(4,1)}  & = &  1,\ 1,\ 3,\ 2,\ 5,\ 3, \ 7, \ 4,\ 9,\ldots.
\end{eqnarray}
The sequences satisfy
\begin{eqnarray}
a_{2i+1}^{(1,4)}  = 2i+1; & \ & a_{2i}^{(1,4)} =4i\\
a_{2i+1}^{(4,1)}  = 2i+1; & \ & a_{2i}^{(4,1)} =i
\end{eqnarray}
and the associated lecture hall sequences satisfy
\[
\GG_{2k}^{(1,4)}: \ \ 
\frac{\la_1}{4k} \geq 
\frac{\la_2}{2k-1} \geq 
\frac{\la_3}{4(k-1)} \geq 
\frac{\la_4}{2k-3} \geq 
\ldots \geq
\frac{\la_{2k}}{1} \geq 0;
\]
\[
\GG_{2k}^{(4,1)}: \ \ 
\frac{\la_1}{k} \geq 
\frac{\la_2}{2k-1} \geq 
\frac{\la_3}{k-1} \geq 
\frac{\la_4}{2k-3} \geq 
\ldots \geq
\frac{\la_{2k}}{1} \geq 0.
\]
By Theorem \ref{thm:k_ell},
\begin{eqnarray}
G_{2k}^{(1,4)}(x,y) \  = \prod_{i=1}^k \frac{1}{(1-x^{2i-1}y^{i-1})(1-x^{4i}y^{2i-1})}
\label{eq:14}
\end{eqnarray}
and
\begin{eqnarray}
G_{2k}^{(4,1)}(x,y) \  = \prod_{i=1}^k \frac{1}{(1-x^{2i-1}y^{4i-4})(1-x^{i}y^{2i-1})}.
\label{eq:41}
\end{eqnarray}
Note that  in $\GG_{2k}^{(1,4)}$, the ratio of consecutive  positive terms $\la_i/\la_{i+1}$ for $i$ odd is at least 
$(4j)/(2j-1) > 2$ for some $j$ or, if $i$ is even,  at least $(2j-1)/(4(j-1)) > 1/2$ for some $j$.
Conversely, every finite sequence $\la$ of positive integers with $\la_i/\la_{i+1} > 2$ for odd $i$ and
 $\la_i/\la_{i+1} > 1/2$ for even $i$ will be in $\GG_{2k}^{(1,4)}$ for $k$ large enough (appending zeroes to $\la$, if necessary).
Thus  $\lim_{k \rightarrow \infty} \GG_{2k}^{(1,4)}$ is the set of sequences (of any even length) satisfying
\begin{equation}
\frac{\la_1}{2} > 
\frac{\la_2}{1} >
\frac{\la_3}{2}>
\frac{\la_4}{1}>
\frac{\la_5}{2}>
\frac{\la_6}{1}> \ldots
\label{2121}
\end{equation}
so the generating function for the sequences $\la$ satisfying \eqref{2121} is given by taking the limit in \eqref{eq:14}:
\begin{equation}
\lim_{k \rightarrow \infty} G_{2k}^{(1,4)}(x,y) = 
\frac{1}{(x;x^2y)_{\infty}(x^4y;x^4y^2)_{\infty}}.
\label{14gf}
\end{equation}
Similarly, 
$\lim_{k \rightarrow \infty} \GG_{2k}^{(4,1)}$ is the set of sequences (of any even length) satisfying
\begin{equation}
\frac{\la_1}{1} > 
\frac{\la_2}{2} >
\frac{\la_3}{1}>
\frac{\la_4}{2}>
\frac{\la_5}{1}>
\frac{\la_6}{2}> \ldots
\label{1212}
\end{equation}
whose generating function is found by taking the limit in \eqref{eq:41}:
\begin{eqnarray}
\lim_{k \rightarrow \infty} G_{2k}^{(4,1)}(x,y) = 
\frac{1}{(x;x^2y^4)_{\infty}(xy;xy^2)_{\infty}}.
\label{41gf}
\end{eqnarray}
In \cite{CortSavSills}, with Corteel and Sills,
we used combinatorial methods to directly compute sum-form generating functions for sequences satisfying
\eqref{2121} and \eqref{1212}.  These generating functions must  be equal to  those in \eqref{14gf} and \eqref{41gf} respectively,
with $x=q$, $y=q$.  In this way we discovered the 
 new pair of identities
whose connection to lecture hall partitions is hidden.
\begin{theorem}[Corteel, S, Sills \cite{CortSavSills}]
\begin{eqnarray}
\sum_{j=0}^{\infty}q^{j(3j-1)/2} \frac{(q^2;q^6)_j}{(q;q)_{3j}} & = & \frac{1}{(q;q^3)_{\infty}(q^5;q^6)_{\infty}}\\
\sum_{j=0}^{\infty}q^{j(3j+1)/2} \frac{(q^4;q^6)_j}{(q;q)_{3j+1}} & = & \frac{1}{(q^2;q^3)_{\infty}(q;q^6)_{\infty}}
\end{eqnarray}
\end{theorem}
Another consequence of \eqref{eq:14} and \eqref{eq:41}  is the following, which also does not mention lecture hall partitions.
\begin{theorem}[S, Sills \cite{SS}]
The number of partitions $\la$ of an integer $N$ into distinct parts $\la_1 > \la_2 > \ldots$ such that 
all parts $\la_{2i}$ are even is equal to the number of partitions of $N$ into parts congruent to $1,5,{\rm or}  \ 6 \ 
(\bmod \ 8)$.
\label{new14}
\end{theorem}
\begin{proof}
The map sending a sequence $\la$ satisfying \eqref{2121} to the partition
$(\la_1, 2\la_2, \la_3, 2\la_4, \ldots)$ is a bijection so replacing $y$ by $y^2$ in \eqref{eq:14} and then setting $x=y=q$ gives the result.
\end{proof}
Similarly, from \eqref{eq:41} we get the following.
\begin{theorem}[\cite{SS}]
The number of partitions $\la$ of an integer $N$ into distinct parts $\la_1 > \la_2 > \ldots$ such that 
all parts $\la_{2i-1}$ are even is equal to the number of partitions of $N$ into parts congruent to $2,3,{\rm or}  \ 7 \ 
(\bmod \ 8)$.
\label{new41}
\end{theorem}

\noindent
(See \cite{CortSavSills} and \cite{SS} for refinements
and  connections with Heine's $q$-Gauss summation.)

Compare Theorems \ref{new14} and \ref{new41}  to the identities below, known as ``G{\"o}llnitz's  little partition theorems''  \cite{gollnitz}.

\begin{theorem}[G{\"o}llnitz \cite{gollnitz}]
The number of partitions of $N$ into parts differing by at least 2 and no consecutive odd parts equals
the number of partitions of $N$ into parts congruent to $1,5,{\rm or}  \ 6 \ 
(\bmod \ 8)$.
\label{little14}
\end{theorem}

\begin{theorem}[G{\"o}llnitz \cite{gollnitz}]
The number of partitions of $N$ into parts differing by at least 2, no consecutive odd parts, and no ones equals
the number of partitions of $N$ into parts congruent to $2,3,{\rm or}  \ 7 \ 
(\bmod \ 8)$.
\label{little41}
\end{theorem}

\noindent
{\bf Question}:  Is there a simple bijection between the partitions  in Theorems \ref{new14} and \ref{little14}?  Between the partitions  in Theorems \ref{new41} and \ref{little41}?

Theorems \ref{new14} and \ref{new41} have been significantly extended and generalized in a recent paper by Berkovich and Uncu \cite{berkovich}.

\section{The refined lecture hall theorem}
\label{refined}

The first proof of the lecture hall theorem in the original paper of Bousquet-M\`elou and Eriksson was derived from Bott's formula for the Poincar\'e series of the affine Weyl group $\widetilde{C}_n$.
In their third paper on lecture hall partitions, Bousquet-M\`elou and Eriksson  proved the following, which was suggested by a refinement of Bott's formula.
\begin{theorem}[{\bf The refined lecture hall theorem} (Bousquet-M\`elou, Eriksson \cite{BME3}]
\begin{align}
L_n(q,u,v) = \sum_{\la \in \LL_n} q^{|\la|} u^{|\ceil{\la}|} v^{o(\ceil{\la})}\ &  = \  
\prod_{i=1}^{n} \frac{1+uvq^i}{1-u^2q^{n+i}}
\label{eq:refined_lhpgf}
\end{align}
where $\ceil{\la}=(\ceil{\la_1/1}, \ceil{\la_2/2}, \ldots, \ceil{\la_n/n})$ and $o(\la)$ is the number of odd parts of $\la$.
\label{thm:refined_lhpgf}
\end{theorem}

In Section \ref{lhp_and_perms}, which establishes a connection between lecture hall partitions and permutations, it  will be seen that  $|\ceil{\la}|$ on lecture hall partitions corresponds to a Mahonian statistic on permutations and $\ceil{\la_n/n}$ corresponds to an Eulerian statistic.

An analog of Theorem \ref{thm:refined_lhpgf}  holds for anti-lecture hall compositions:
\begin{theorem}[{\bf The refined anti-lecture hall theorem} (Corteel, S  \cite{anti}]
\begin{align}
A_n(q,u,v) = \sum_{\la \in \anti_n} q^{|\la|} u^{|\floor{\la}|}v^{o(\floor{\la})} \ &  = \  
\prod_{i=1}^{n} \frac{1+uvq^i}{1-u^2q^{i+1}}
\label{eq:refined_antigf}
\end{align}
where $\floor{\la}=(\floor{\la_1/1}, \floor{\la_2/2}, \ldots, \floor{\la_n/n})$.
\label{them:refined_antigf}
\end{theorem}

The generating functions \eqref{Ltrunc} and \eqref{Atrunc} for truncated lecture hall partitions and truncated anti-lecture hall compositions can be similarly refined, as shown in \cite{truncated}.

\begin{theorem}[Corteel, S \cite{truncated}]
\begin{align}
{\overline L}_{n,k}(q,u,v)
\ &  = \ (uv)^kq^{{k+1 \choose 2}}
\dqbin{n}{k}{q} \frac{((-u/v)q^{n-k+1};q)_k}{(u^2q^{2n-k+1};q)_k};
\label{eq:refined_lhpgf_trunc}\\
A_{n,k}(q,u,v) 
\ &  = \  
(uv)^kq^{{k+1 \choose 2}}
\dqbin{n}{k}{q} \frac{(-uvq^{n-k+1};q)_k}{(u^2q^{2(n-k+1)};q)_k}.
\label{eq:refined_antigf_trunc}
\end{align}
\label{thm:refined_trunc}
\end{theorem}

\noindent
{\bf Question:} Is there a refinement of Bott's formula corresponding to \eqref{eq:refined_lhpgf_trunc}?

Equations \eqref{eq:refined_lhpgf} and \eqref{eq:refined_antigf} (dropping the $v$) imply the relationship
\begin{eqnarray}
L_n(q,u) = A_n(q^{-1},uq^{n+1}).
\label{LAreciprocity}
\end{eqnarray}
The search for a combinatorial explanation for the relationship \eqref{LAreciprocity} motivated some of the work in the next few sections.
A geometric view proved successful in \cite{BS}.

\section{An analog of the Gaussian polynomials for lecture hall partitions}
\label{box}

The generating function for  $P_{n \times k}$, the set of partitions  with at most $n$ parts and with largest part at most $k$, is given by the Gaussian polynomial:
\begin{eqnarray}
\sum_{p \in P_{n \times k} }q^{|p|}  & =  &\dqbin{n+k}{k}{q}.
\label{Gaussian}
\end{eqnarray}
We get an interesting result if we similarly consider the lecture hall partitions in $\LL_n$ with largest part at most $k$.
Write $k$ as $k=tn+i$ where $0 \leq i < n$.
\begin{theorem}[Corteel, Lee, S \cite{CLS}]
\begin{eqnarray}
|\{\la \in \LL_n \ | \ \la_n \leq tn +i\}|  & = & (t+1)^{n-i}(t+2)^i.
\label{eq:box}
\end{eqnarray}
\end{theorem}
For fixed $n$, this is 
a polynomial in $k$ for fixed $i$.
In particular, when $i=0$ we have:
\begin{eqnarray}
|\{\la \in \LL_n \ | \  \frac{\la_n}{n} \leq t\}| &= &(t+1)^{n}.
\label{cubegf}
\end{eqnarray}
A similar result holds for anti-lecture hall compositions \cite{CLS}:
\begin{eqnarray}
|\{\la \in \anti_n \ | \  \frac{\la_n}{1} \leq t\}| &=& (t+1)^{n}
\label{anti_box}
\end{eqnarray}
because of the following.
\begin{theorem}[Corteel, Lee, S \cite{CLS}]
For any sequence $\s$ of positive integers,
\begin{eqnarray}
|\{\la \in \LL^{(s_1, \ldots, s_n)}_n \ | \  \frac{\la_n}{s_n} \leq t\}| &= &
|\{\la \in \LL^{(s_n, \ldots, s_1)}_n \ | \  \frac{\la_n}{s_1} \leq t\}|
\label{reverse}
\end{eqnarray}
\label{thm:reverse}
\end{theorem}
The fact that $(t+1)^{n}$ is the number of points in the unit cube suggests  a connection with permutations.  In establishing the relationship between lecture hall partitions and
 permutations in the next section,
the following $u$-analog of \eqref{eq:box}, in which $u$  tracks the statistic $|\ceil{\la}|$  from Section \ref{refined}, will be relevant.
\begin{theorem}[S, Schuster \cite{MR2881231}]
\begin{eqnarray}
\sum_{\la \in \{\la \in \LL_n |  \la_n \leq tn +i\}} u^{|\ceil{\la}|}   & = &( \qint{t+1}_u)^{n-i}(\qint{t+2}_u)^i.
\label{eq:qbox}
\end{eqnarray}
\label{thm:qbox}
\end{theorem}
Only recently, with Corteel and Lovejoy, were we able to give $(q,u,v)$-refinements of equations \eqref{cubegf} and \eqref{anti_box} \cite{CLovejoyS}.

\section{Lecture hall partitions and permutations}
\label{lhp_and_perms}

In this section we relate statistics on lecture hall partitions to Eulerian and Mahonian statistics on permutations.

The {\em Eulerian polynomial}, $E_n(x)$, defined by
\begin{eqnarray}
\sum_{t \geq 0} (t+1)^n x^t & = & \frac{E_n(x)}{(1-x)^{n+1}},
\label{Enx}
\end{eqnarray}
has the combinatorial interpretation 
\begin{eqnarray}
E_n(x) & = & \sum_{\pi \in \Sn} x^{\des \pi}
\label{Enx_des}
\end{eqnarray}
where $\Sn$ is the set of permutations of $1, \ldots, n$ and for $\pi \in \Sn$, 
$\D \pi$ is the set of {\em descents} of $\pi$, i.e. the set of positions $i$ for which
$\pi_i > \pi_{i+1}$, and $\des \pi = |\D \pi|$.

The {\em major index} of  $\pi \in \Sn$ is the statistic $\maj \pi =\sum_{i \in \D \pi} i $.  The joint distribution of $\des$ and $\maj$ over permutations has the following form.  
This is a special case of a result of MacMahon.
 \begin{theorem}[MacMahon, \cite{macmahon}, Vol. 2. p. 211]
\begin{eqnarray}
\sum_{t \geq 0} (\qint{t+1}_u)^n x^t & = & \frac{\sum_{\pi \in \Sn} x^{\des \pi} u^{\maj \pi}}{\prod_{i=0}^n(1-xu^i)}.
\label{Enxu}
\end{eqnarray}
\label{thm:Enxu}
\end{theorem}
If we group lecture hall partitions according to the size of the last part as follows:
\begin{eqnarray*}
\LL_n & =&  \bigcup_{t\geq 0} \left \{ \la \in \LL_n \ \bigm\vert  \ t-1 < \frac{\la_n}{n} \leq t \right \}
\end{eqnarray*}
 the union is disjoint and applying Theorem \ref{thm:qbox} with $i=0$ we have
\begin{equation}
\sum_{\la \in \LL_n} x^{\ceil{\la_n/n}}u^{|\ceil{\la}|}\  = \ \sum_{t \geq 0} ((\qint{t+1}_u)^n-(\qint{t}_u)^n )x^t  \ = \  (1-x)\sum_{t \geq 0} (\qint{t+1}_u)^n x^t.
\label{xceil_dist}
\end{equation}
Combining \eqref{Enxu}, and \eqref{xceil_dist} gives the following.
\begin{theorem}[S, Schuster \cite{MR2881231}]
\begin{eqnarray}
\sum_{\la \in \LL_n} x^{\ceil{\la_n/n}} u^{|\ceil{\la}|} \ &= &\ \frac{\sum_{\pi \in \Sn} x^{\des \pi}u^{\maj \pi}}{\prod_{i=1}^{n}(1-xu^i)}.
\label{eqn:permstats}
\end{eqnarray}
\label{thm:permstats}
\end{theorem}
We show in Section \ref{generalize} that Theorem \ref{thm:permstats} can be further refined and can be extended to describe statistics on $\s$-lecture hall partitions.

What statistic on permutations corresponds to the weight statistic  $|\la|$ on lecture hall partitions?
We consider this in the next section and again in Section \ref{generalize}.

\section{Quadratic permutation statistics}
\label{quadratic}

Returning to the observation \eqref{LAreciprocity} in Section \ref{refined},
with Bright, in \cite{BS}  we found bijections $\LL_n \rightarrow \integers^n$ and $\anti_n \rightarrow \integers^n$ which, when combined, proved that $L_n(q,u) = A_n(q^{-1},uq^{n+1})$.  

In doing so we connected both lecture hall partitions and anti-lecture hall compositions to permutations and discovered some interesting statistics. The inversion number, $\inv \pi$ of a permutation $\pi \in \Sn$ is defined by
\[\inv \pi = | \{ (i,j) \ | \ i < j \ {\rm and} \ \pi_i > \pi_j \}|.
\]
Over $\Sn$, $\inv$ has the same distribution as $\maj$.
 Define  new statistics $\bin \pi$ and $\sq \pi$ by
$$\bin  \pi = \sum_{i \in \D \pi} {i+1 \choose 2} ; \ \ \ \ \sq \pi =  \sum_{i \in \D \pi}  i^2.$$
Both $\bin$ and $\sq$ are {\em quadratic} permutation statistics, in that they are defined by a sum 
of {\em quadratic} functions of $i$ over descent positions $i$.  

It seems that quadratic statistics are not well-known, although they have appeared in the context of Weyl group generating functions (\cite{stembridge},
\cite{Weyl}, \cite{zabrocki}).  They show up in the following theorem.

\begin{theorem}[Bright, S \cite{BS}]
\[
A_n(q,u)   = \sum_{\la \in \anti_n} q^{|\la|} u^{|\floor{\la}|} =\frac{\sum_{\pi \in \Sn}q^{\bin \pi + \inv \pi}u^{\maj \pi}}{\prod_{i=1}^n(1-u^iq^{i(i+1)/2})}.
\]
\label{BSanti}
\end{theorem}
This suggests merging $\bin$ and $\inv$ into a single statistic 
\[
\binv \pi = \bin \pi + \inv \pi.
\]
One can then combine Theorem \ref{BSanti}  with Theorem \ref{them:refined_antigf} to
compute the joint distribution of the quadratic statistic $\binv$ and the linear statistic $\maj$ over permutations.
\begin{corollary}[Bright, S \cite{BS}]
\[
\sum_{\pi \in \Sn} u^{\maj \pi} q^{\binv \pi} = \prod_{i=1}^n (1-u^iq^{i(i+1)/2}) \frac{1+uq^i}{1-u^2q^{i+1}}.
\]
\label{majbinv}
\end{corollary}
Setting $q=1$  in Corollary \ref{majbinv} we get the well-known distribution of major index:
\[
\sum_{\pi \in \Sn} u^{\maj \pi} = \prod_{i=1}^n \qint{i}_u.
\]
Setting $u=1$  in  Corollary \ref{majbinv}  shows that $\binv$ itself has an interesting generating function.
\[
\sum_{\pi \in \Sn}  q^{\binv \pi} = \prod_{i=1}^n \qint{2}_{q^i}\frac{\qint{i(i+1)/2}_q}{\qint{i+1}_q}.
\]
Replacing $q$  by $q^{-1}$ and then setting $u=q^{n+1}$ in Corollary  \ref{majbinv} gives a a nice result:
\begin{eqnarray}
\sum_{\pi \in \Sn}  q^{(n+1)\maj \pi -\binv \pi} & =& \prod_{i=1}^n \qint{i}_{q^{2(n-i)+1}}
\label{lhpdist}
\end{eqnarray}
and this product will appear again in Section \ref{generalize}.
We get an even simpler joint distribution with $\maj$ if we consider the statistic
\[
\sqin \pi = \sq \pi + \inv \pi.
\]
The following was given a simple combinatorial proof  in \cite{BS} which, unlike Corollary  \ref{majbinv},
did not require a refined lecture hall theorem.
\begin{theorem}[Bright, S \cite{BS}]
\[
\sum_{\pi \in \Sn} u^{\maj \pi} q^{\sqin \pi} = \prod_{i=1}^n \qint{i}_{uq^i}.
\]
\label{majsqin}
\end{theorem}
These distributions seem to have gone mostly unnoticed although the special case of Theorem
\ref{majsqin} when $q=q^{-1}$ and $u=q^n$
was proven by Zabrocki in \cite{zabrocki}.
Recently,  quadratic permutation statistics arose in work of Paul Johnson on $q$-rational Catalan numbers \cite{Johnson}.  He required the following identity, which he proved in \cite{Johnson}.
\begin{corollary} [Johnson \cite{Johnson}]
\[
\sum_{\pi \in \Sn} u^{\maj \pi} q^{\siz \pi} = \prod_{i=1}^n \qint{i}_{uq^{n+1-i}},
\]
where  $\siz \pi = (n+1) \maj \pi - \sqin \pi$.
\label{cor:Johnson}
\end{corollary}
Note that replacing $q$ by $q^{-1}$ and then setting $u=uq^{n+1}$ in Theorem \ref{majsqin} gives  Corollary 
\ref{cor:Johnson}.

\section{Lattice point generating functions of lecture hall cones}
\label{sec:lpt}

For a sequence $\s$ of positive integers, the {\em $\s$-lecture hall cone} is defined as
\begin{align}
\C_n^{(\s)} =  \left \{\la \in \reals^n \ \biggm\vert \  0 \leq \frac{\la_{1}}{{s_1}}
\leq \frac{\la_{2}}{{s_2}} \leq \cdots
\leq \frac{\la_{n}}{{s_n}}  \right \}.
\end{align}
Then
$
\LL_n^{(\s)}  = \C_n^{(\s)}  \cap \integers^n.
$
The {\em lattice point generating function} of $\C_n^{(\s)}$ is
\[
F_n^{(\s)}(z) = \sum_{\la \in \LL_n^{(\s)}} z^{\la},
\]
where $z^{\la}=z_1^{\la_1}z_2^{\la_2}\cdots z_n^{\la_n}$. 
$L_n^{(\s)}(q)$ is obtained by setting each $z_i=q$.

The cone $\C_n^{(\s)}$ is generated by the vectors $\{\vv_1, \ldots, \vv_n\}$ where $\vv_i=[0, \ldots, 0,s_i, \ldots, s_n]$:
\[
\C_n^{(\s)} = \left \{\sum_{i=1}^n \alpha_i\vv_i \ | \  \alpha_i \geq 0 \right \}.
\]
The  (half open) {\em fundamental  parallelepiped} associated with $\{\vv_1, \ldots, \vv_n\}$ is
\[
\Pi_n^{(\s)} = \left \{\sum_{i=1}^n \alpha_i\vv_i \ | \ 0 \leq \alpha_i < 1 \right \}.
\]
It is well known (see for example \cite{Barvinok} p. 40)  that 
$F_n^{(\s)}(z)$ can be expressed as follows.
\begin{theorem}
\[
F_n^{(\s)}(z) = \frac{\sum_{\la \in \Pi_n^{(\s)} \cap \integers^n}z^{\la}}{\prod_{i=1}^n(1-z^{\vv_i})}.
\]
\label{lptgf}
\end{theorem}

So, enumerating all $\s$-lecture hall partitions reduces to enumerating those in $\Pi_n^{(\s)}$.
There are $s_1s_2 \cdots s_n$  lattice points in $\Pi_n^{(\s)}$ and  
we will see in Section \ref{generalize}  that they can be characterized in terms of inversion sequences and their statistics.

\noindent
{\bf Example:} 
$\C_2^{(2,3)}$ has generators $[2,3]$ and $[0,3]$ and 
$$\Pi_n^{(\s)}\cap \integers^n = \{(0,0),(0,1),  (0,2),  (1,2),  (1,3),  (1,4) \},$$ so

\[
F_2^{(2,3)}(z) = \frac{1+z_2+z_2^2+z_1z_2^2+z_1z_2^3+z_1z_2^4}{(1-z_1^2z_2^3)(1-z_2^3)}
= \frac{1+z_1z_2^2}{(1-z_1^2z_2^3)(1-z_2)}
\]
and
\[
L_2^{(2,3)}(q) = \frac{1+q^3}{(1-q)(1-q^5)}.
\]

Note that we could replace the vector $\vv_n=[0,0, \ldots, s_n]$ by $\vv'_n=[0,0, \ldots, 1]$.
This would generate the same cone, but the fundamental parallelepiped would be different.
More about this in Section \ref{parallel}.

\section{Lecture hall polytopes and Ehrhart theory}
\label{ehrhart}

In Section \ref{lhp_and_perms} statistics on lecture hall partitions are described in terms of statistics on permutations.  We show that
it is possible to describe $\s$-lecture hall partitions in terms of statistics on a generalization of permutations.

The {\em $s$-lecture hall polytope} is the bounded region of the $s$-lecture hall cone defined by
\begin{align}
\PP_n^{(\s)} =  \left \{\la \in \reals^n \ \large | \  0 \leq \frac{\la_{1}}{{s_1}}
\leq \frac{\la_{2}}{{s_2}} \leq \cdots
\leq \frac{\la_{n}}{{s_n}} \leq 1 \right \}.
\end{align}
The $t$-th dilation of 
$\PP_n^{(\s)}$ is $t\PP_n^{(\s)}= \{t\la \ | \ \la \in \PP_n^{(\s)}\}$.

The {\em Ehrhart polynomial} of $\PP_n^{(\s)}$ is the number of lattice points in its $t$-th dilation:
\[
i_n^{(\s)}(t) = |t\PP_n^{(\s)} \cap \integers^n|.
\]
The vertices of $\PP_n^{(\s)}$ are $\{(0,0,\ldots,0,s_i,s_{i+1} \ldots, s_n \ | \ 1 \leq i \leq n+1\}$
and since all coordinates are integers, $i_n^{(\s)}(t)$ is guaranteed to be a polynomial in $t$
\cite{ehrhart1,ehrhart2}.

The {\em Ehrhart series} of $\PP_n^{(\s)}$ is $\sum_{t \geq 0} i_n^{(\s)}(t) x^t$ and it is known that 
\begin{eqnarray}
\sum_{t \geq 0} i_n^{(\s)}(t) x^t \ &= &\ \frac{E_n^{(\s)}(x)}{(1-x)^{n+1}},
\label{ehrhart}
\end{eqnarray}
where $E_n^{(\s)}(x)$ is a polynomial with nonnegative integer coefficients (see, e.g., \cite{stanley}).
$E_n^{(\s)}(x)$ is called the {\em $h^*$-polynomial} of  the polytope $\PP_n^{(\s)}$.

The relationship between the  $\s$-lecture hall partitions and $E_n^{(\s)}(x)$, analogous to the derivation of \eqref{xceil_dist} in Section \ref{lhp_and_perms}, is then given by
\begin{eqnarray}
\sum_{\la \in \LL_n^{(\s)} }x^{\ceil{\la_n/s_n}} & = & (1-x)\sum_{t \geq 0} i_n^{(\s)}(t) x^t \ \ = \ \ \frac{E_n^{(\s)}(x)}{(1-x)^{n}}.
\label{lhtoEn}
\end{eqnarray}
When $\s=(1,2, \ldots, n)$, 
by
 \eqref{cubegf},  
 
\[
i_n^{(1,2, \ldots, n)}(t)  \ =  \ (t+1)^n 
\]
and thus by \eqref{Enx} and \eqref{ehrhart},
$E_n^{(1,2, \ldots, n)}(x)$ is the Eulerian polynomial $E_n(x)$.

For this reason, we refer to $E_n^{(\s)}(x)$ the {\em $\s$-Eulerian polynomial}. 
In the next section we provide a combinatorial interpretation of $E_n^{(\s)}(x)$ for general $\s$.

\section{Generalized lecture hall partitions and  $\s$-inversion sequences}
\label{generalize}

A permutation $\pi \in \Sn$ can be encoded as the {\em inversion sequence}
$\phi(\pi)=(e_1, \ldots, e_n)$ where
\[
e_i = \left |  \{j \ | \ j<i \ {\rm and} \ \pi_j > \pi_i\} \right |.
\]
Then $i \in \D(\pi)$ if and only if $e_i < e_{i+1}$ and 
$\phi: \Sn \rightarrow \{(e_1, \ldots, e_n) \ | \ 0 \leq e_i < i\}$ is a bijection.
Inversion sequences will be used to generalize permutations.

Given a sequence of positive integers $\s$,  define the {\em $\s$-inversion sequences} by
$$\I_n^{(\s)} = \left \{ \e \in \integers^n \ | \ 0 \leq e_i < s_i \right \}.$$ 
For $\e \in \I_n^{(\s)}$ define the {\em ascent set} of $\e$ by
\[
\A \e = \left \{ i \ \biggm\vert  \  0 \leq i < n \ {\rm and} \ \frac{e_i}{s_i} < \frac{e_{i+1}}{s_{i+1}} \right \},
\]
with the convention that $e_0=0$ and $s_0=1$.   Let $\asc \e = |\A \e|$.

So, for example, $\e=(3,2) \in \I_2^{(5,3)}$ has $\A (3,2) = \{0,1\}$ since $ 0 < e_1/s_1=3/5$ and
$e_1/s_1 = 3/5 < 2/3 = e_2/s_2$.

In view of \eqref{lhtoEn}, the following theorem provides a combinatorial characterization  of the 
$\s$-Eulerian polynomial.
\begin{theorem}[S, Schuster \cite{MR2881231}]
For any sequence $\s$ of positive integers,
\begin{eqnarray}
\sum_{\la \in \LL_n^{(\s)}} x^{\ceil{\la_n/s_n}} &=&
\frac{\sum_{\e \in \I_n^{(\s)}} x^{\asc \e}}{(1-x)^{n}}.
\label{bars}
\end{eqnarray}
\label{invseq}
\end{theorem}
\begin{proof}
We outline the proof of Theorem \ref{invseq} from \cite{MR2881231}, which adapts the idea of {\em barred permutations} from Gessel and Stanley \cite{GS}.

A {\em barred inversion sequence}  is an inversion sequence  $\e \in \I_n^{(\s)}$ with ``bars'' between consecutive elements.  It can  have any number of bars before $e_1$ or between $e_i$ and $e_{i+1}$,  but 
if $i \in \A \e$, at least one bar must occur between 
$e_i$ and $e_{i+1}$.

We show that both sides of \eqref{bars} count all barred inversion sequences of $\I_n^{(\s)}$ according to the number of bars.

First, fix the inversion sequence $\e \in \I_n^{(\s)}$ and consider all possible barrings of $\e$.  There must be at least $\asc \e$ bars and then any number of additional bars can be distributed into the $n$ possible positions.  Summing over all $\e\in \I_n^{(\s)}$ gives the right-hand side of \eqref{bars}.

For the second way, we show the following is a bijection between lecture hall partitions with $\ceil{\la_n/s_n}=t$ and barred inversion sequences with $t$ bars.
Let $\la \in  \LL_n^{(\s)}$ with $\ceil{\la_n/s_n}=t$.
Let
\[
b = \left( \ceil{\la_1/s_1}, \ceil{\la_2/s_2},  \ldots, \ceil{\la_n/s_n}   \right).
\]
Then $b_1 \leq b_2 \leq \ldots \leq b_n=t$.  Let $\e = (e_1, \ldots, e_n)$, where
\[
e_i = s_ib_i-\la_i.
\]
Clearly, $\e \in \I_n^{(\s)}$, since $0 \leq e_i < s_i$.  Insert $t$ bars into $\e$ by placing $b_1$ bars before $e_1$ and
$b_i-b_{i-1}$ bars before $e_i$ for $2 \leq i \leq n$. It can be checked that since $\la \in \LL_n^{(\s)}$,
if there is no bar between $e_i$ and $e_{i+1}$ then $b_i=b_{i+1}$ and
$e_i/s_i \geq e_{i+1}/s_{i+1}$ so that $i \not \in \A \e$.

Then summing the lecture hall partitions with $\ceil{\la_n/s_n}=t$ over all $t$, as in the left-hand side of \eqref{bars} also gives the number of barred inversion sequences.
\end{proof}

From Theorem \ref{invseq} and \eqref{lhtoEn} we have a characterization of the $\s$-Eulerian polynomials.
\begin{corollary}
For any sequence $\s$ of positive integers $\s$
\[
E_n^{(\s)}(x)  =\sum_{\e \in \I_n^{(\s)}} x^{\asc \e}.
\]
That is, the $h^*$-polynomial of the $s$-lecture hall polytope is the  {\em ascent polynomial} for the $\s$-inversion sequences.
\label{Echarac}
\end{corollary}
There is more about the $s$-Eulerian polynomials in the next section, but we note now that
Theorem \ref{invseq} can be refined to track other statistics.
Define the following statistics on $\s$-inversion sequences $\e$:
\[
\amaj \e = \sum_{i \in \A e}(n-i)
\]
\[
\lhp \e = -|\e| + \sum_{i \in \A e}(s_{i+1} + \cdots + s_n).
\]
Define the following statistics on $\s$-lecture hall partitions $\la$:
\[
\ceil{\la} = [\ceil{\la_1/s_1}, \ldots, \ceil{\la_n/s_n}]
\]
\[
\epsilon^+(\la) = [s_1\ceil{\la_1/s_1}-\la_1, \ldots, s_n\ceil{\la_n/s_n}-\la_n]
\]
Following the statistics on a lecture hall partition as it maps to a barred inversion sequence in the proof of Theorem \ref{invseq} will give this 4-parameter refinement.
\begin{theorem}[S, Schuster \cite{MR2881231}]
For any sequence $\s$ of positive integers,
\[
\sum_{\la \in \LL^{(\s)}_n}q^{|\la|} x^{\ceil{\la_n/s_n}} u^{|\ceil{\la}|}z^{|\epsilon^+(\la)|} \ = \ \frac{\sum_{\e \in \I_n^{(\s)}} x^{\asc \e}u^{\amaj \e}q^{\lhp \e} z^{|\e|}}{\prod_{i=0}^{n-1}(1-xu^{n-i}q^{s_{i+1}+ \cdots + s_n})}.
\] 
\label{fullss}
\end{theorem}
In \cite{MR2881231}, Theorem \ref{fullss} is applied in various ways to obtain known and new enumeration results.  We mention just one here.

Setting $\s=(1, \ldots, n)$ gives the following connection between lecture hall partitions and permutations, extending the results in Section \ref{lhp_and_perms}.
\begin{corollary}
\[
\sum_{\la \in \LL_n} q^{|\la|}u^{\ceil{\la}}x^{\ceil{\la_n/n}}z^{|\e|} =\frac{\sum_{\pi \in \Sn} x^{\des \pi}u^{\comaj \pi}q^{\lhp \pi}z^{\inv \pi}}{\prod_{i=0}^{n-1}(1-xu^{n-i}q^{(i+1) + \ldots + n})},
\] 
where for $\pi \in \Sn$,  $$\comaj  \pi = \sum_{i \in \D \pi} (n-i)$$ and
$$\lhp \pi = -\inv \pi + \sum_{i \in \D \pi} ((i+1) + \ldots + n)).$$
\label{permsgf}
\end{corollary}
The statistic $\lhp$ in Corollary  \ref{permsgf} is another appearance of a quadratic permutation statistic.  As shown in \cite{MR2881231}, setting $x=z=1$ in Corollary  \ref{permsgf} and combining with
the refined lecture hall theorem gives the joint distribution of $\lhp$ and $\comaj$:
\begin{corollary} [S, Schuster \cite{MR2881231}]
\[
\sum_{\pi \in \Sn} q^{\lhp \e} u^{\comaj \e} \ = \ \prod_{k=1}^n
\frac{(1+uq^k)(1-u^{n+1-k}q^{k+ \ldots + n})}{1-u^2q^{n+k}}.\]
\label{uqcor}
\end{corollary}
Setting $u=1$ gives a simple generating polynomial for the lecture hall statistic on permutations,
the same as \eqref{lhpdist}.
\begin{corollary} [S, Schuster \cite{MR2881231}]
\[
\sum_{\pi \in \Sn} q^{\lhp \e}  \ = \ \prod_{k=1}^n \qint{k}_{q^{2(n-k)+1}}.
\]
\label{qcor}
\end{corollary}
Note that by Theorem \ref{fullss},
Corollary \ref{uqcor} is {\em equivalent} to the refined lecture hall theorem and Corollary \ref{qcor}
is equivalent to the lecture hall theorem.  Is there a simple combinatorial proof of Corollary
\ref{qcor}?

\section{$\s$-Eulerian polynomials}
\label{ascentpoly}
\begin{table}
\begin{center}
{\small
\begin{tabular} {|l|c|c|}
\hline
&sequence $\s$ & $\s$-Eulerian polynomial\\
\hline \hline
&& \\
(i) &$(1,2,3,4,5,6)$ &
$1+57\,{x}+302\,{x}^{2}+302\,{x}^{3}+57\,x^4+x^5$\\ 
&& \\
\hline
&& \\
(ii) & $(6,5,4,3,2,1)$ &
$1+57\,{x}+302\,{x}^{2}+302\,{x}^{3}+57\,x^4+x^5$
\\& & \\
\hline
&& \\
(iii) &  $(2,4,6,8,10)$ &
$1+237\,{x}+1682\,{x}^{2}+1682\,{x}^{3}+237\,x^4+x^5$
\\ && \\
\hline
&& \\
(iv) &$(1,3,5,7,9,11)$ &
$1 + 358\, x + 3580\, x^2 + 5168\,x^3 + 1328\,x^4 + 32\, x^5$. 
\\
&& \\
\hline

&& \\
(v) &$(1,4,3,8,5,12)$ & 
$1 + 209\,x+1884\,{x}^{2}+2828\,{x}^{3}+811\,{x}^{4}+27\,{x}^5$
\\ && \\
\hline&& \\
(vi) &$(1,1,3,2,5,3)$ &
$1 +  20\,x + 48\,{x}^2 + 20\,x^3 + x^4$
\\ && \\
\hline
&& \\
(vii) & $(7,2,3,5,4,6)$ &
$1 + 71\, x + 948\, x^2 + 2450\, x^3 + 1411\, x^4 + 159 \, x^5$
\\
&& \\
\hline
\end{tabular}
}
\end{center}
\caption{$\s$-Eulerian polynomials}
\label{polys}
\end{table}

By Corollary \ref{Echarac} in the previous section, $E_n^{(\s)}(x)$ is the ascent polynomial of the $\s$-inversion sequences.
 Table \ref{polys} shows $E_n^{(\s)}(x)$ for various sequences $\s$.
 
 Row (i) of the table contains the Eulerian polynomial $E_6(x)$. 
 The sequences in  rows (i) and (ii) of the table give rise to the same Eulerian polynomial.
 This is true in general.
 \begin{theorem}[S, Schuster \cite{MR2881231}]
  For any sequence $\s$ of positive integers,
\[
E_n^{(s_1, s_2, \ldots, s_n)}(x) = E_n^{(s_n, s_{n-1}, \ldots, s_1)}(x).
\]
 \end{theorem}
 \begin{proof}
 From Theorem \ref{thm:reverse}, 
 $\PP_n^{(s_1, s_2, \ldots, s_n)}$ and  $\PP_n^{(s_n, s_{n-1}, \ldots, s_1)}$ have the same Ehrhart polynomial, so the result follows from  \eqref{ehrhart}.
 \end{proof}
 
 The polynomial in row (iii) of Table \ref{polys} might be recognized as the Eulerian polynomial for {\em signed permutations}:
 Let
 $B_n = \{(\sigma_1, \ldots, \sigma_n) \ | \  \exists \pi \in \Sn, \  \forall   i  \ \sigma_i = \pm \pi(i) \} $
 and let
 $\des \sigma =| \{i \in \{0, \ldots, n-1\} \ | \ \sigma_i > \sigma_{i+1} \}|$,
with the convention that $\sigma_0 = 0$.  
 \begin{theorem} [Pensyl, S \cite{PS1}]
$$\sum_{\sigma \in B_n} x^{\des \sigma} = E_n^{(2,4, \ldots, 2n)}(x).$$
\label{Bn}
\end{theorem}
This was proved bijectively in \cite{PS1} where it was refined to include Mahonian and other statistics.
In fact, the idea extends to  the wreath product
$\Sn \wr \integers_k$  with the natural notion of descent.  
 \begin{theorem} [Pensyl, S \cite{PS2}]
$$\sum_{\sigma \in \Sn \wr \integers_k} x^{\des \sigma} = E_n^{(k,2k, \ldots, nk)}(x).$$
\label{wreath}
\end{theorem}
Observe that for $k \geq 1$, $\LL_n^{(k,2k, \ldots, nk)}= \LL_n$.
Thus by Theorems \ref{invseq} and  \ref{wreath},
\[
\sum_{\la \in \LL_n}x^{\ceil{\la_n/(nk)} } \ =\ 
\frac{\sum_{\sigma \in \Sn \wr \integers_k} x^{\des \sigma} }{(1-x)^n}.
\]
We highlight this surprising result below.
\begin{theorem}[Pensyl, S \cite{PS2}]
The distribution of the descent statistic on the the wreath product $\Sn \wr \integers_k$ is given by the distribution of the statistic $\ceil{\la_n/(nk)}$ over lecture hall partitions $\la_n \in \LL_n$ {\rm (}multiplied by a factor of $(1-x)^n$.{\rm )}
This includes permutations, when $k=1$,  and signed permutations, when $k=2$.
\end{theorem}

The {\em $1/k$-Eulerian polynomial} $E_{n,k}(x)$  of \cite{SavVis} is the ascent polynomial of the ``$1 \bmod k$''
inversion sequences, that is:
$$E_{n,k}(x)\ =\ E_n^{(1,k+1,2k+1, \ldots, (n-1)k+1)}(x).$$
In Table \ref{polys}, the polynomial in row (iv) is
$E_{6,2}(x)=E_6^{(1,3,5,7,9,11)}(x).$
The $1/k$-Eulerian polynomial derives  its name from the following.
\begin{theorem}[S, Viswanathan \cite{SavVis}]
For $k \geq 1$,
\[
 \sum_{n \geq 0} E_{n,k}(x) \frac{z^{n}}{n!}  \ = \ 
\left ( \frac{1-x}{e^{kz(x-1)}-x}\right ) ^{\frac{1}{k}}.
\]
\end{theorem}
The $1/k$-Eulerian polynomials are related to permutations as follows.
\begin{theorem}[S, Viswanathan \cite{SavVis}]
\begin{eqnarray*}
E_{n,k}(x)=\sum_{e \in \I_n^{(1,k+1,2k+1, \ldots, (n-1)k+1)}} x^{\asc e } &  =  &
\sum_{\pi \in \Sn} x^{\exc\, \pi }  k^{n-\cyc\, \pi },
\end{eqnarray*}
where
$\exc\, \pi  = \card{\{i \ | \ \pi(i) > i\}}$
and
$\cyc\, \pi $ is the number of cycles in the disjoint cycle representation
of $\pi$.
\label{exccyc}
\end{theorem}
The number of $1 \bmod k$-inversion sequences is  $|\I_n^{(1,k+1,2k+1, \ldots, (n-1)k+1)}| = \prod_{i=0}^{n-1}(ik+1)$, which is the same as the number of $k$-Stirling permutations of order $n$.
It is natural to ask if there is a statistic for $k$-Stirling permutations whose distribution is $E_{n,k}(x)$.
In \cite{MM},  Ma and Mansour
show that $E_{n,k}(x)$ gives the distribution of a statistic called ``ascent plateaus'' on the $k$-Stirling permutations of order $n$.   We are not aware of a bijective proof for this result or for Theorem
\ref{exccyc}.

The  sequence  in row (v) of Table \ref{polys} 
is the $(4,1)$-sequence discussed in Section \ref{Gollnitz}.
Although the cardinality of $\I_{2n}^{(1,1,3,2,\ldots, 2n-1,n)}$ is the same as the number of permutations of $\{1,1,2,2, \ldots, n,n \}$, it comes as a surprise that the ascent generating function for the first set {\em is} the descent polynomial for the second.
\begin{theorem} [S, Visontai \cite{SV}]
$E_{2n}^{(1,1,3,2,\ldots, 2n-1,n)}(x)$ is the descent polynomial for  permutations of the  multiset $\{1,1,2,2, \ldots, n,n \}$.
\label{multi1}
\end{theorem}
To prove Theorem \ref{multi1}, we made use of the Ehrhart polynomial of  $\PP_{2n}^{(1,1,3,2,\ldots, 2n-1,n)}$, which had been computed in \cite{MR2881231}.  We
 then used \eqref{ehrhart} to get an explicit expression for $E_{2n}^{(1,1,3,2,\ldots, 2n-1,n)}(x)$.
This matched MacMahon's generating function for the descent polynomial of multiset permutations.
A combinatorial proof would be quite interesting.

The sequence in Table \ref{polys}, row (vi) is the $(1,4)$-sequence. 
The following was conjectured in \cite{SV}, and was proved independently by 
Lin in \cite{Lin} and  by Chen et al in \cite{Chenetal}.
\begin{theorem}[Lin \cite{Lin} and  Chen et al \cite{Chenetal}]
$E_{2n}^{(1,4,3,8,\ldots, 2n-1,4n)}(x)$ is the descent polynomial for the
signed permutations of  $\{1,1,2,2, \ldots, n,n \}$.
\end{theorem}

\section{Real-rooted polynomials}
\label{roots}
The example in the last row of Table \ref{polys} illustrates that the coefficient sequence of $E_n^{(\s)}(x)$ can be unimodal, even though the sequence $\s$ is not.  
Recall that $E_n^{(\s)}(x)$ is the $h^*$-polynomial of the polytope $\PP_n^{(\s)}$. However, the
 coefficient sequence of the $h^*$-polynomial  of a convex lattice polytope need not be unimodal. 
 It turns out that $\PP_n^{(\s)}$ is special.

It is known that if a polynomial has all real roots, then its sequence of coefficients is unimodal and log-concave.
In \cite{SV} we proved the following with Visontai, using the method of compatible polynomials developed by Chudnovsky and Seymour  in \cite{ChudnovskySeymour}.
\begin{theorem}[S, Visontai \cite{SV}]
For any sequence $\s$ of positive integers,  the $\s$-Eulerian polynomial has only real roots.
\label{thm:roots}
\end{theorem}

It then follows as a corollary of Theorem \ref{thm:roots}  that all of the polynomials discussed in the previous section are real-rooted.
It has been known since Frobenius \cite{Fro} that the Eulerian polynomials are real-rooted.
It was first shown by Brenti that the type-B Eulerian polynomials have only real roots \cite{brenti94}.
The real-rootedness of the descent polynomials of the wreath products  $\Sn \wr \integers_k$ was shown by 
Steingr\'{i}msson
\cite{Stein}.
In \cite{brenti2000} Brenti proved  that the  polynomial $\sum_{\pi \in \Sn} x^{\exc \pi} y^{\cyc \pi}$
 has all real roots for any positive $y$ \cite{brenti2000}.  This implies that the $1/k$-Eulerian polynomials have only real roots.
Simion \cite{simion}  established the real-rootedness of descent polynomials of multiset permutations.

In \cite{brenti94}, Brenti conjectured that the descent polynomial of any finite Coxeter group has only real roots.
This was known to be true for all except the type-$D$ Coxeter groups, $D_n$.
$D_n$  consists of signed permutations in $B_n$ in which an even number of elements are negative.
Descents in $D_n$ are defined as follows, but now with the convention that  $\sigma_0 = -\sigma_2$.
\[
{\rm Des}_D  \, \sigma = \left \{ i \ |  \  0 \leq i < n \ {\rm and}  \ \sigma_i > \sigma_{i+1} \right \}.
\] 
In \cite{SV} with Visontai, we were able to use inversion sequences to settle the the type-$D$ case and thereby settle Brenti's conjecture.
\begin{theorem}[S, Visontai \cite{SV}]
The Eulerian polynomials of type $D$ have only real roots.
\label{type D}
\end{theorem}
To prove Theorem \ref{type D},
it can be checked that the generating polynomial for the type-$D$ descent statistic over $B_n$ is twice the descent polynomial of $D_n$.
By correspondingly defining a type-$D$ ascent statistic for the $(2,4, \ldots, 2n)$-inversion sequences,  the technique of the proof of Theorem \ref{thm:roots} can be modified to prove Theorem \ref{type D}.

In \cite{Branden}, Br{\"a}nd{\'e}n translates the proof of Theorem \ref{type D} in \cite{SV} into the language of Coxeter groups.

Inversion sequences were used  in \cite{SV} to settle or make progress on several other conjectures, including the following.   Dilks, Peterson, and Stembridge
defined a notion of ``affine descent'' for an irreducible affine Weyl group $W$.  The {\em affine Eulerian polynomial} of $W$ is then the generating polynomial for this statistic over elements of $W$.
They conjectured that  the affine Eulerian polynomials for all finite Weyl groups are real-rooted.
It was known for all but types $B$ and $D$.  It was settled  for type $B$  in \cite{SV} and Yang and Zhang settled it for type $D$ in \cite{YZ} to complete the proof.
\begin{theorem}[Dilks, Petersen, Stembridge \cite{dilks}, S, Visontai \cite{SV}, Yang, Zhang \cite{YZ}]
The affine Eulerian polynomials for all finite Weyl groups are real-rooted.
\end{theorem}

\section{Generating lecture hall partitions by height}
\label{height}

The Ehrhart series for the lecture hall polytope $\PP_n^{(\s)}$ gives rise to the Eulerian polynomials
$E_n^{(\s)}(x)$  which have been surprisingly useful at modeling combinatorial statistics.
However, $E_n^{(\s)}(x)$ is the generating polynomial for $\la \in \LL_n^{(\s)}(x)$ according to the unusual statistic $\ceil{\la_n/s_n}$.
What about $\sum_{\la \in \LL_n^{(s)}} x^{\la_n}$, which would seem more natural?

Instead of working with the lecture hall polytope $\PP_n^{(\s)}$ consider the rational lecture hall polytope in which the bounding condition is $\la_n \leq 1$ rather than $\la_n/s_n \leq 1$:
\begin{align}
\R_n^{(\s)} =  \left \{\la \in \reals^n \ \biggm\vert  \  0 \leq \frac{\la_{1}}{{s_1}}
\leq \frac{\la_{2}}{{s_2}} \leq \cdots
\leq \frac{\la_{n}}{{s_n}} \leq \frac{1}{s_n} \right \}.
\end{align}
Then the height generating function for lecture hall partitions is given by
\begin{eqnarray}
\sum_{\la \in \LL_n^{(\s)}}x^{\la_n} = (1-x) 
\sum_{t \geq 0}  |t\R_n^{(\s)} \cap \integers^n| x^t.
\label{ht}
\end{eqnarray}
The function 
\[
t \rightarrow  |t\R_n^{(\s)} \cap \integers^n|
\]
is no longer guaranteed to be a polynomial, as it was for $\PP_n^{(\s)}$, but rather a quasi-polynomial.
The Ehrhart series will have the form
\[
\sum_{t \geq 0}  |t\R_n^{(\s)} \cap \integers^n| x^t = \frac{Q_n^{(\s)}(x)}{(1-x)(1-x^{s_n})^n},
\]
where $Q_n^{(\s)}(x)$ is a polynomial with nonnegative integer coefficients.

In \cite{PS1}, with Pensyl,  we prove that  $Q_n^{(\s)}(x)$ has the following interpretation in terms of inversion sequences.  This gives the height generating function for $\LL_n^{(\s)}$.
$Q_n^{(\s)}(x)$, called the {\em inflated $s$-Eulerian polynomial}, is an ``inflated'' version of $E_n^{(\s)}(x)$.
\begin{theorem}[Pensyl, S \cite{PS1}]
For any sequence $\s$ of positive integers,
\[
\sum_{\la \in \LL_n^{(\s)}} x^{\la_n} \ =  \frac{Q_n^{(\s)}(x)}{(1-x^{s_n})^n} = \frac{ \sum_{\e \in \I_n^{(\s)}} x^{s_n \asc \e - e_n}}{(1-x^{s_n})^n}.
\]
\label{ht_gf_x}
\end{theorem}

In contrast to the $\s$-Eulerian polynomial, $Q_n^{(\s)}(x)$ is not  real-rooted.
Is the  coefficient sequence of $Q_n^{(\s)}(x)$  unimodal for all positive integer sequences $\s$? (It seems to be from our computations.)

We can get an explicit expression for the height generating function of $\LL_n$.
\begin{corollary}[Pensyl, S \cite{PS1}]
For the original lecture hall partitions, $\LL_n$,
\[
\sum_{\la \in \LL_n} x^{\la_n} \ = \  \frac{\sum_{\pi \in \Sn} x^{n\,\des \pi -n + \pi_n}}{(1-x^n)^n}
\ = 
(1-x) \sum_{j \geq 0} \sum_{i=0}^{n-1} (j+1)^{n-i}(j+2)^i x^{jn+i}.
\]
\label{lhp_ht_gf}
\end{corollary}
\begin{proof}
Set $\s=(1, \ldots, n)$ in Theorem \ref{ht_gf_x}.  For the first equality,  note that the mapping $\phi:  \Sn \rightarrow \I_n$
defined by $\phi(\pi) = (e_1, \ldots e_n)$, where
$e_i = |\{j>0 \ | \ j< i \ {\rm and} \ \pi_j > \pi_i\}$, has the property that $\des \pi = \asc \phi(\pi)$ and $e_n = n-\pi_n$.

For the second equality,  we have  computed the Ehrhart quasi-polynomial in \eqref{eq:box}: 
$$|(jn+i)\R_n \cap \integers^n| = (j+1)^{n-i}(j+2)^i $$ and  so the result follows by \eqref{ht}.
\end{proof}

The polynomials $\sum_{\pi \in \Sn} x^{n\, \des \pi + \pi_n}$ have appeared before, for example in 
 work of  Chung and Graham on inversion-descent polynomials \cite{ChungGraham2013}, where the following is shown:
\begin{theorem}[Chung, Graham \cite{ChungGraham2013}]
\[
\frac{\sum_{\pi \in \Sn} x^{n\, \des \pi + \pi_n}}{1+x+x^2 + \ldots + x^{n-1}} \ = \ 
\sum_{\pi \in {\mathbf S}_{n-1}} x^{n \,\des \pi + \pi_{n-1}}.
\]
\label{ChungGraham}
\end{theorem}
Theorem \ref{ChungGraham} implies that 
$Q_n^{(1, \ldots, n)}(x)$ 
is divisible by
$1+x+ \ldots x^{n-1}$.  It can be shown that a similar result holds for any $\s$ and there is a nice combinatorial characterization of the quotient.
\begin{theorem}[Auli, S \cite{AS}]
For any sequence $\s$ of positive integers,
\begin{eqnarray*}
\frac{Q_n^{(\s)}(x)}{1+x+ \ldots x^{s_{n}-1}} & = &
\sum_{\e \in \I_{n-1}^{(\s)}} x^{s_n \asc \e - \floor{s_ne_{n-1}/s_{n-1}}}
\end{eqnarray*}
\label{Qdivided}
\end{theorem} 

Finally, we note that Theorem \ref{ht_gf_x} has a refinement analogous to Theorem \ref{fullss}.
\begin{theorem}[Pensyl, S \cite{PS1}]
For any sequence $\s$ of positive integers,
\[
\sum_{\la \in \LL_n^{(\s)}}q^{|\la|} x^{\la_n}u^{|\ceil{\la}|}z^{|\epsilon^+(\la)|} \ 
 = \frac{ \sum_{\e \in \I_n^{(\s)}} x^{s_n \asc \e - e_n}u^{\amaj \e}q^{\lhp \e} z^{|e|}}
 {\prod_{i=0}^{n-1}(1-x^{s_n}u^{n-i}q^{s_{i+1}+ \cdots +s_n})}.
\]
\end{theorem}

\section{Fundamental parallelepipeds}
\label{parallel}

It is helpful to return to Section \ref{sec:lpt} and view the polynomials
$E_n^{(\s)}(x)$,
 $Q_n^{(\s)}(x)$, and
$Q_n^{(\s)}(x)/(1+x+ \ldots + x^{s_{n}-1})$
in terms of fundamental parallelepipeds associated with generating sets   of the  lecture hall cone
$\C_n^{(\s)}(x)$. 

Let $V_n(\s)= \{\vv_1, \ldots, \vv_n\}$ be the generating set for $\C_n^{(\s)}$  where $\vv_i=[0, \ldots, 0,s_i, \ldots, s_n]$.
Let  $V'_n(\s)= \{\vv_1, \ldots, \vv_{n-1}\} \cup \{[0,\ldots,0,1]\}$, also a generating set for $\C_n^{(\s)}$.
Let $\Pi_n(\s)$ and $\Pi'_n(\s)$ be the (half open) fundamental parallelepipeds associated with $V_n(\s)$ and $V'_n(\s)$, respectively.

Liu and Stanley \cite{LiuS} first observed that under the bijection of Theorem \ref{invseq}, the inverse image of the minimally barred inversion sequences (which are counted by the numerator  in the right-hand side of \eqref{bars}) is exactly those $\la \in \Pi_n(\s)$.
So we have, by Theorem \ref{invseq}:
\begin{equation}
\sum_{\la \in \Pi_n(\s) \cap \integers^n} x^{\ceil{\la_n/s_n}} \ = \  E_n^{(\s)}(x) \ = \ \sum_{\e \in \I_n^{(\s)}} x^{\asc \e}.
\end{equation}
Setting $z_n=x$ and $z_1= \cdots= z_{n-1} = 1$ in Theorem \ref{lptgf} and comparing to Theorem \ref{ht_gf_x} we have:
\begin{equation}
\sum_{\la \in  \Pi_n(\s) \cap \integers^n} x^{\la_n} \ = \  Q_n^{(\s)}(x) \ = \ \sum_{\e \in \I_n^{(\s)}} x^{s_n \asc \e - e_n}.
\end{equation}
On the other hand, if we consider the generating set $V_n'(\s)$ and reinterpret Theorem \ref{lptgf} in conjunction with Theorem \ref{Qdivided} we get:
\begin{equation}
\sum_{\la \in \Pi'_n(\s) \cap \integers^n} x^{\la_n} \ = \  
\frac{Q_n^{(\s)}(x)}{1+x+ \ldots + x^{s_n-1}} \ = \ \sum_{\e \in \I_{n-1}^{(\s)}} x^{s_n \asc \e - \floor{s_ne_{n-1}/s_{n-1}}}.
\end{equation}
Chung and Graham observed in \cite{ChungGraham2013}
the following consequence of Theorem \ref{ChungGraham}:  
the sequence of nonzero coefficients of 
$Q_n^{(1, \ldots, n)}(x)/(1+x+ \ldots x^{n-1})$ coincides with that of
$Q_{n-1}^{(1, \ldots, n-1)}(x)$.  This generalizes as follows.
\begin{theorem}[Auli, S \cite{AS}]
For every nondecreasing sequence $\s$ of positive integers,
the sequence of nonzero coefficients of 
$Q_n^{(\s)}(x)/(1+x+ \ldots + x^{s_{n}-1})$ coincides with that of 
$Q_{n-1}^{(\s)}(x)$.
\end{theorem}

Note that for such $\s$, this means that
the set $ \Pi'_n(\s) \cap \integers^n$ can be constructed from
 $\Pi_{n-1}(\s) \cap \integers^{n-1}$
by lifting each $\la=(\la_1, \ldots \la_{n-1})$ to 
$(\la_1, \ldots \la_{n-1}, \ceil{\s_n\la_{n-1}/s_{n-1}})$.
And the set $\Pi_n(\s) \cap \integers^n$
can be constructed from the set $\Pi'_n(\s) \cap \integers^n$ by mapping each 
$\la=(\la_1, \ldots \la_{n})$ to the $s_{n}$ points
$(\la_1, \ldots \la_{n}+i)$ for $0 \leq i < s_n$.

A complete characterization of such sequences is given in \cite{AS}.

\section{Gorenstein cones}
\label{Gcones}

A pointed rational cone $C \subseteq \reals^n$ is {\em Gorenstein} if there exists a lattice point $c$ in the interior $C^0$ of $C$ such that
\[
C^0 \cap \integers^n = c + (C \cap \integers^n).
\]
For example, $\C_2^{(3,5)}$ is Gorenstein, but $\C_2^{(5,2)}$  is not.

A rational function $H(w_1, \ldots, w_k)$  is {\em self-reciprocal} if
\[
H(1/w_1, \ldots, 1/w_k)= (\pm1) w_1^{d_1} \ldots w_k^{d_k}H(w_1, \ldots, w_k)
\]
 for some integers
$d_1, \ldots, d_k$.

A wonderful property of  Gorenstein cones is that the generating function for its lattice points according to {\em any} ``proper'' grading is self-reciprocal.  A {\em proper grading} of $C$ is a function
$g:  C \rightarrow {\mathbb N}^r$, for some $r$, satisfying 
(i) $g(\la+ \mu)=g(\la)+g(\mu)$; (ii) $g(\la)=0$ implies $\la = 0$; and (iii) for any $v \in {\mathbb N}^r$,
$g^{-1}(v)$ is finite.

\noindent
Examples of proper gradings on $\C_n^{(\s)} $ are
\begin{itemize}
\item
$\la \rightarrow ({\la_1}, \ldots, {\la_n})$ 
\item
$\la \rightarrow {\la_n}$ 
\item
$\la \rightarrow {|\la|}.$
\end{itemize}
However, for general $\s$,  $\la \rightarrow \ceil{\la}$ is not proper nor is
$\la \rightarrow {\ceil{\la_n/s_n}}$.

For 
$m= (m_1, \ldots, m_r) \in {\mathbb N}^r$ 
let $X^{m}=X_1^{m_1} \cdots X_r^{m_r}$. 
The following is a special case of a result due to Stanley \cite{StanleyG}.
\begin{theorem}[Stanley \cite{StanleyG}]
Let $g:  C \rightarrow {\mathbb N}^r$ be a proper grading of a pointed rational cone $C$ and let $H$ be the generating function
\[
H(X_1, \ldots, X_r) =  H(X)= \sum_{\la \in C} X^{g(\la)}.
\]
Then $C$ is Gorenstein if and only if $H$ is self-reciprocal.
\label{thm:StanleyG}
\end{theorem}

\noindent
So in Gorenstein lecture hall cones $\C_n^{(\s)}(x)$, all of the following are self-reciprocal:
\begin{itemize}
\item
The lattice point generating function $F_n^{(\s)}(z)$.
\item
The generating function $L_n^{(\s)}(q)$.
\item
The Ehrhart series of the polytope $\R_n^{(\s)}$ (but not necessarily of  $\PP_n^{(\s)}$).
\item
The inflated $\s$-Eulerian polynomial $Q_n^{(\s)}(x)$ (but not necessarily the $\s$-Eulerian polynomial $E_n^{(\s)}(x)$).
\item The polynomial $Q_n^{(\s)}(x)/(1+x+ \ldots + x^{s_n-1}).$
\end{itemize}
And if the cone is not Gorenstein then none of these can be self-reciprocal.

So when are lecture hall cones Gorenstein?
This brings us back to Section \ref{slhp} and  the original question of  Bousquet-M\'elou and Eriksson \cite{BME2} regarding polynomic sequences: Which sequences $\s$  have $L_n^{(\s)}(q)=
\prod_{i=1}^n(1-q^{d_i})^{-1}$?
Clearly, in order for  $\s$ to be polynomic, $L_n^{(\s)}(q)$
must be self-reciprocal,
 as Bousquet-M\'elou and Eriksson pointed out, and therefore by Theorem \ref{thm:StanleyG} the cone must be Gorenstein.  And by Theorem \ref{thm:StanleyG}  this is guaranteed if the lattice point generating  function $F_n^{(\s)}(z)$ is self-reciprocal.  
 
 Gorenstein $\s$-lecture hall cones can be characterized by simple arithmetic conditions on the sequence $\s$.

\begin{theorem}[Beck et al \cite{Gorenstein}, Bousquet-M{\'e}lou, Eriksson \cite{BME2}]
For any sequence $\s$ of positive integers, $\C_n^{(\s)} $ is Gorenstein if and only if there is a $c \in \integers^n$ satisfying
\[
c_js_{j-1}=c_{j-1}s_j + \gcd(s_j,s_{j-1})
\]
for $j>1$ with $c_1=1$.
\label{Gcondition}
\end{theorem}
The $(k,\ell)$-sequences satisfy the conditions of Theorem \ref{Gcondition} as do the sequences $(1,k+1, 2k+1, \ldots)$.
The Fibonacci sequence does not satisfy it for $n \geq 5$.

Beck et al showed in \cite{Gorenstein}  that  the $\ell$-sequences are unique among sequences defined by second order linear recurrences in the following sense.

\begin{theorem} [Beck et al \cite{Gorenstein}]
Let $\s$ be a sequence of positive integers defined by
\[
s_n = \ell s_{n-1}+ms_{n-2},
\]
with $s_1=1, s_2=\ell$, where $\ell >0$ and either $m>0$ or $0 < |m| <\ell$.
Then $\C_n^{(\s)} $ is Gorenstein for all $n\ge 1$ if and only if $m=-1$.
(That is, if and only if $\s$ is an $\ell$-sequence.)  If $m \not = -1$, there exists $n_0=n_0(\ell,m)$ such that $\C_n^{(\s)} $ fails to be Gorenstein for all $n \geq n_0$.
\label{only_ell}
\end{theorem}
As a consequence of Theorem \ref{only_ell}, no other sequence defined by a second order linear recurrence of this form can be polynomic except perhaps for finitely many values of $n$.

The proof of Theorem \ref{only_ell} is quite involved when $\gcd(s_i,s_{i+1}) > 1$.  Perhaps there is a simpler approach?

\section{Concluding remarks}
\label{conclusion}
This report is by no means comprehensive, but rather gives a sampling of results, techniques, and connections in the developing theory of lecture hall partitions.  A recent result, for example, shows that  the lecture hall cone for $\LL_n$ has a unimodular triangulation \cite{triangulate}.
The geometry has led to a lot of interesting results, observations, and connections.
But there is still no transparent proof of the lecture hall theorem and, as yet, no geometric proof.

\vspace{.2in}
\noindent
{\bf Acknowledgments} I would like to thank all of my collaborators on the lecture hall projects.
Thanks to AIM for hosting a SQuaRE in Polyhedral Geometry and Partition Theory, to the IMA for hosting a Workshop  on Geometric and Enumerative Combinatorics,
and to the Simons Foundation for travel support for collaboration. I am grateful to the referees for their corrections and suggestions to improve the manuscript.

I am particularly indebted to Herbert Wilf.  I first learned about integer partitions at his invited address at the 1988 SIAM Discrete Math Conference in San Francisco.  That changed everything.

\bibliographystyle{plain}
\bibliography{mlhp}

\end{document}